%% file: vniqDel10.tex
\documentclass[11pt,a4paper,reqno]{amsart}
 \usepackage{epstopdf}
\usepackage{amsmath,amsthm,amsfonts,amssymb,bm,wasysym}
\usepackage{comment}
\usepackage{caption}
\usepackage{subcaption}
\usepackage{fullpage,bbm}
\usepackage{graphicx}
\usepackage{tikz,ifthen}
\usepackage{hyperref}
\usepackage{psfrag}
\usetikzlibrary{calc}
\usetikzlibrary{patterns}
\usetikzlibrary{arrows}
\usetikzlibrary{decorations.pathreplacing,decorations.pathmorphing, patterns,shapes}
\theoremstyle{plain}
\newtheorem{theorem}{Theorem}
\newtheorem{proposition}[theorem]{Proposition}
\newtheorem{corollary}[theorem]{Corollary}
\newtheorem{lemma}[theorem]{Lemma}

\theoremstyle{definition}

\theoremstyle{remark}
\newtheorem{remark}{Remark}

\begin{document}

\title{Coalescence of Euclidean geodesics on the Poisson-Delaunay triangulation}
\def\MLine#1{\par\hspace*{-\leftmargin}\parbox{\textwidth}{\[#1\]}}

\def\d{{\rm d}}
\def\Del{\mathsf{Del}}
\def\Env{\mathsf{Env}}
\def\E{\mathbb{E}}
\def\es{\emptyset}
\def\intt{\mathsf{int}}
\def\ellv{l^{\mathsf{v}}}
\def\ellh{l^{\mathsf{v}}}
\def\Ld{L_{\mathsf{d}}}
\def\Lu{L_{\mathsf{u}}}
\def\mc{\mathcal}
\def\ms{\mathsf}
\def\one{\mathbbmss{1}}
\def\Q{\mathbb{Q}}
\def\P{\mathbb{P}}
\def\Pu{P_{\ms{u}}}
\def\Pm{P_{\ms{m}}}
\def\Pd{P_{\ms{d}}}
\def\R{\mathbb{R}}
\def\Rng{\mathsf{Rng}}
\def\Z{\mathbb{Z}}
\def\T{\mathcal{T}}
\def\F{\mathcal{F}}
\def\eps{\varepsilon}

\begin{abstract}
Let us consider Euclidean first-passage percolation on the Poisson-Delaunay triangulation. We prove almost sure coalescence of any two semi-infinite geodesics with the same asymptotic direction. The proof is based on an adapted Burton-Keane argument and makes use of the concentration property for shortest-path lengths in the considered graphs. Moreover, by considering the specific example of the relative neighborhood graph, we illustrate that our approach extends to further well-known graphs in computational geometry. As an application, we show that the expected number of semi-infinite geodesics starting at a given vertex and leaving a disk of a certain radius grows at most sublinearly in the radius.
\end{abstract}


\author{David Coupier}
\address{Laboratoire Paul Painlev\'e U.M.R.~CNRS 8524, Universit\'e Lille 1, France.}
\email{david.coupier@math.univ-lille1.fr}
\author{Christian Hirsch}
\thanks{D.C. was supported by the CNRS GdR 3477 GeoSto and the Labex CEMPI (ANR-11-LABX-0007-01)}
\address{Weierstrass Institute for Applied Analysis and Stochastics Berlin, Germany.}
\email{christian.hirsch@wias-berlin.de}
\thanks{C.H.~was supported by the Leibniz program
\emph{Probabilistic Methods for Mobile Ad-Hoc Networks}}

\maketitle

\section{Introduction and main results}
\label{defSec}

\subsection{Introduction}

In the seminal papers~\cite{lNewman,newmSurf} Licea and Newman have shown coalescence of a large class of geodesics with the same direction in the standard two-dimensional lattice first-passage percolation (FPP) model. To be more precise, for Lebesgue-almost every direction, any two geodesics with this direction starting from different initial points coalesce with probability $1$. Later, Howard and Newman have investigated an isotropic first-passage percolation model based on the homogeneous Poisson point process in the plane~\cite{efpp,efpp2}. Due to the property of isotropy, they could derive an analogous coalescence statement for \emph{any} fixed direction. 

As of today, the coalescence of geodesics has been extended to further isotropic models of continuum percolation. For instance, in~\cite{pimentel1} a Euclidean FPP model is considered on the Poisson-Delaunay triangulation, where the edge weights are independent and identically distributed according to an absolutely continuous distribution on $[0,\infty)$. In the present paper, we extend the coalescence result to the case where the edge weights are dependent and given by the Euclidean length $\nu_1(e)$ of the corresponding edge $e$. Due to the strong dependence between edge weights and the underlying graph structure, local modification arguments that have been successfully applied in~\cite{pimentel1} do not carry over to the setting of Euclidean edge weights. We therefore develop a novel adaptation of the Burton-Keane argument. In our approach, a central role is played by a modified FPP model, where geodesics are forbidden to backtrack behind certain vertical lines. Moreover, we show that our general approach is not restricted to the Delaunay triangulation, but can also be applied to other graphs of interest in computational geometry. As a specific example, we provide explicitly the adaptations that are needed to deal with the Poisson relative neighborhood graph.\\

	In order to state our main results precisely, we first recall the definitions of the Delaunay triangulation and the relative neighborhood graph and refer the reader to~\cite{jarom} for further properties. Let us start with a set of vertices in $\R^2$ given by a homogeneous Poisson point process $X$ whose intensity is assumed to be equal to $1$. Then, the \emph{Delaunay triangulation $\Del(X)$ on $X$}, is the geometric graph on the vertex set $X$, where an edge is drawn between two vertices $x,y\in X$ if and only if there exists a disk whose intersection with the vertex set $X$ consists precisely of the points $x$ and $y$. The \emph{relative neighborhood graph $\Rng(X)$ on $X$}, is the geometric graph on the vertex set $X$ where an edge is drawn between two vertices $x,y\in X$ if and only if there does not exist a vertex $z\in X$ such that $\max\{|x-z|,|y-z|\}<|x-y|$. In particular, almost surely, the relative neighborhood graph is a subgraph of the Delaunay triangulation. In the following, many results hold for both $\Del(X)$ and $\Rng(X)$. Hence, we write $G(X)$ as notation to represent either of these graphs. Also, letters $x,x',y,z$ will refer to elements of $X$.

Next, let us introduce a Euclidean FPP model on $G(X)$ and explain the notion of geodesics. Let $P,P'\in\R^2$ be points that are contained on the edge set of $G(X)$. Then, we denote by $\ell(P,P')=\ell_{G(X)}(P,P')$ the Euclidean length $\nu_1(\gamma)$ of the shortest Euclidean path $\gamma$ on $G(X)$ connecting $P$ and $P'$. That is,
\begin{equation}
\label{defgeoFPP}
\ell(P,P') = \inf\{ \nu_1(\gamma) :\, \gamma \text{ is a path on $G(X)$ connecting $P$ and $P'$} \} ~.
\end{equation} 
For any path $\gamma$ on $G(X)$ and $P,P'\in\gamma$, we write $\gamma[P,P']$ for the subpath of $\gamma$ starting at $P$ and ending at $P'$. When the (possibly infinite) path $\gamma$ satisfies $\ell(P,P')=\nu_1(\gamma[P,P'])$ for all $P,P'\in\gamma$, then $\gamma$ is called \emph{geodesic}.

The present paper investigates geodesics $\gamma$ on $G(X)$ that are \emph{semi-infinite} in the sense that $\gamma$ emanates from a certain starting point but consists of infinitely many vertices. Moreover, writing $S^1=\{P\in\R^2:\,|P|=1\}$ for the unit circle, we say that a semi-infinite path $\gamma$ on $G(X)$ admits an \emph{asymptotic direction} $\hat{u}\in S^1$ if and only if
$$
\lim_{\substack{|y|\to\infty\\ y\in\gamma}} \frac{y}{|y|} = \hat{u} ~.
$$

Let $\hat{u}\in S^1$ be an arbitrary direction. Based on the classical arguments developed in~\cite{efpp,lNewman}, it is proved in~\cite{modDev} that almost surely, for every point $x\in X$, there exists a unique semi-infinite geodesic starting at $x$ and with asymptotic direction $\hat{u}$. Using the terminology of~\cite{lNewman}, this geodesic is called \emph{$\hat{u}$-unigeodesic} and will be denoted by $\gamma_x$ in the following. Note that, $\gamma_x$ obviously depends on the direction $\hat{u}$, but since in our paper this direction will always be clear from the context, we adhere to the simplified notation. 


\subsection{Coalescence of geodesics}

The first main result of our paper establishes the coalescence of $\hat{u}$-unigeodesics.

\begin{theorem}
\label{delThm}
Consider Euclidean FPP on $G(X)$. Then, for any given direction $\hat{u}\in S^1$, with probability 1 any two geodesics with asymptotic direction $\hat{u}$ eventually coalesce. That is, $\gamma_x\cap \gamma_{x'}\ne\es$ for all $x,x'\in X$.
\end{theorem}

On a very general level, the proof of Theorem~\ref{delThm} is based on the Burton-Keane technique that has emerged as a powerful tool in the analysis of random planar trees~\cite{dsf2,efpp,lNewman, pimentel1}. However, the implementation of the Burton-Keane argument for Euclidean FPP on the Poisson-Delaunay triangulation $\Del(X)$ and the Poisson relative neighborhood graph $\Rng(X)$ is markedly different from the examples that have been discussed in the literature so far. Let us explain why.

The classical argument of Burton-Keane starts by assuming that there exist $\hat{u}$-unigeodesics that do not coalesce. Then, a local modification argument is used to show that each of these geodesics has a positive probability to be surrounded by a protective shield. This shield prevents that distant geodesics coalesce with it. In particular, the expected number of shielded -- and therefore non-coalescent -- geodesics in a bounded box grows as the box area. This contradicts the fact that the expected number of edges of $G(X)$ crossing the box boundary grows as the box perimeter, i.e., as the square root of the box area. 
In a given FPP model, the difficulty of carrying out this program lies in the local modification step.

For instance, in the setting of iid edge weights~\cite{lNewman,pimentel1}, a simple modification of edge weights makes it possible to generate shields that are avoided by the geodesics starting outside such obstacles. Indeed, by choosing sufficiently large weights, these geodesics are forced to circumvent the shielded area. Conversely, in Euclidean FPP models such as the ones considered in~\cite{efpp} and the present paper, the weights are determined by the locations of the vertices given by a Poisson point process. There is no additional source of randomness on which one could rely. Hence, any local modification step must modify the Poisson point process itself. The classical Euclidean FPP model considered in~\cite{efpp} is based on the complete graph on a Poisson point process where weights are given by certain powers of the Euclidean distance. In particular, the absence of a graph topology entails a powerful monotonicity property: removing Poisson points can only increase shortest-path lengths. This makes it possible to create obstacles by deleting Poisson points in a large region without modifying the geodesics.

The main challenge in analyzing Euclidean FPP on the Delaunay triangulation $\Del(X)$ and the relative neighborhood graph $\Rng(X)$ lies in the lack of a related monotonicity property. Indeed, in both models, the removal of Poisson points has two opposite effects. First, this invalidates paths passing through a deleted vertex, and then increases shortest-path lengths (as in the previous Euclidean FPP models). Second, deleting vertices also has the possibility of \emph{unblocking} certain edges which can potentially decrease shortest-path lengths when they appear. This \emph{self-healing property} makes it much harder to describe the effect of removing points. Therefore, Euclidean FPP on the Delaunay triangulation $\Del(X)$ and the relative neighborhood graph $\Rng(X)$ are markedly different from the FPP models previously considered in the literature~\cite{efpp,lNewman,pimentel1} and requires us to give to the Burton-Keane approach of~\cite{efpp,lNewman} a new twist.

The key idea for the proof of Theorem~\ref{delThm} is to consider a modified Euclidean FPP model in which geodesics are preserved under certain local modifications. More precisely, this modified FPP model forbids geodesics to backtrack behind a given vertical line. This allows us to implement local modifications to the left of that line without influencing geodesics to the right of it. Hence, we are able to construct in the modified FPP model a family of non-coalescent geodesics whose mean number grows as the box area. However, in general these new geodesics are no longer geodesics in the original FPP model.\\ 

In order to illustrate the generality of our approach, we stress that shortest-path lengths on the Delaunay triangulation may behave quite differently to shortest-path lengths on the relative neighborhood graph. Indeed, it is known that the Delaunay triangulation is a \emph{spanner}~\cite{span1,span2}. That is, when constructing the Delaunay triangulation from an arbitrary locally finite set, any two vertices can be connected by a path whose length is at most $4\sqrt{3}\pi/9$ times the Euclidean distance between these vertices. In contrast, starting from an arbitrary infinite set, the relative neighborhood even need not necessarily be connected. Nevertheless, using a Poisson point process as vertex set, Euclidean FPP both on the Delaunay triangulation and the relative neighborhood graph are well-behaved asymptotically. More precisely, as a key tool in the proof of Theorem~\ref{delThm}, we leverage recently established results on concentration of shortest-path lengths on the Poisson-Delaunay triangulation and the Poisson relative neighborhood graph~\cite{modDev}. Even if the spanner-property of the Delaunay triangulation makes it impossible to punish the passage through obstacles by arbitrarily high costs, the strong concentration of shortest-path lengths implies that creating obstacles with moderately high costs is sufficient for achieving the desired shielding effect.

\subsection{Sublinearity of the number of geodesics leaving large disks}

Let $x^{\star}$ be the closest Poisson point to the origin $o\in\R^2$. We denote by $\T_{x^{\star}}$ the collection of geodesics $\gamma_{x^{\star},x}$, for any $x\in X$, where $\gamma_{x^{\star},x}$ is the shortest Euclidean path on $G(X)$ connecting $x^{\star}$ and $x$. The continuous nature of the underlying Poisson point process $X$ ensures the a.s.~uniqueness of geodesics. Hence, $\T_{x^{\star}}$ is a.s.~a tree rooted at $x^{\star}$, called the \textit{shortest-path tree} on the graph $G(X)$ w.r.t.~the root $x^{\star}$.

From the moderate deviations result obtained in \cite[Theorem 2]{modDev} for the geodesics $\gamma_{x,x'}$ (w.r.t.~the straight line segment connecting $x$ and $x'$) and the general method developed in \cite[Proposition 2.8]{efpp2}, the following statement holds with probability $1$: for every direction $\hat{u}\in S^1$, there exists at least one semi-infinite geodesic in $\T_{x^{\star}}$ starting from $x^{\star}$ with asymptotic direction $\hat{u}$. As a consequence, the tree $\T_{x^{\star}}$ a.s.~admits an infinite number of semi-infinite geodesics (in all the directions).

As the second main result of this paper, we show that this (expected) number is asymptotically sublinear. To do it, let us denote by $\chi_{r}$ the number of semi-infinite geodesics in $\T_{x^{\star}}$ crossing the circle $S_r(o)=rS^1$. Since these geodesics may cross various times any given circle, we have to be more precise. Let us consider the graph obtained from $\T_{x^{\star}}$ after deleting any geodesic $(x^{\star},x_{2},\ldots,x_{n})$ with $n\geq 2$ (except the endpoint $x_{n}$) such that the Poisson points $x^{\star},x_{2},\ldots,x_{n-1}$ belong to the disk $B_r(o)=\{P\in\R^2:\,|P|\le1\}$ but not $x_{n}$. Then, $\chi_{r}$ counts the unbounded connected components of this graph. Thus, let us consider a direction $\hat{u}\in S^1$ and a real number $c>0$. Among the previous unbounded connected components, the ones coming from an edge $(x_{n-1},x_{n})$ whose segment $[x_{n-1};x_{n}]$ crosses the arc of the circle $S_r(o)$ centered at $r\hat{u}$ and with length $c$, are counted by $\chi_{r}(\hat{u},c)$. Using this notation, we state:

\begin{theorem}
\label{theo:sublin}
Let $\hat{u}\in S^1$ and $c>0$. Then,
\begin{equation}
\label{sublinFPP}
\lim_{r\to\infty} r^{-1} \E \chi_{r} = 0 \; \mbox{ and } \; \lim_{r\to\infty} \E \chi_{r}(\hat{u},c) = 0 ~.
\end{equation}
\end{theorem}

The first limit of (\ref{sublinFPP}) can be understood as follows. Among all the edges of $\T_{x^{\star}}$ crossing the circle $S_r(o)$, whose mean number is of order $r$, a very small number of them belong to semi-infinite geodesics.

The first limit $r^{-1}\E\chi_{r}\to 0$ immediately follows from the directional result $\E\chi_{r}(\hat{u},c)\to 0$, since by isotropy $\E \chi_{r}=r\E \chi_{r}(\hat{u},2\pi)$, for any $\hat{u}\in S^1$. To state a null limit for the expectation of $\chi_{r}(\hat{u},2\pi)$ we follow the general method developed in \cite{sublin}. This method essentially relies on two ingredients: a local approximation (in distribution) of the tree $\T_{x^{\star}}$ far away from its root by a suitable directed forest and the a.s.~absence of bi-infinite geodesics in this directed forest. Actually, this absence of bi-infinite geodesics is a consequence of the coalescence of all the semi-infinite geodesics having the same asymptotic direction, i.e., Theorem \ref{delThm}.

Observe that thanks to the translation invariance of the graph $G(X)$, Theorem \ref{theo:sublin} remains true 
whatever the Poisson point $x$ of $X$ at which the considered tree is rooted.

Finally, although the product $r^{-1}\chi_{r}$ is conjectured to converge to $0$ almost surely, we can prove following ideas of~\cite[Section 3.2]{sublin} that $\chi_{r}(\hat{u},c)$ does not tend to $0$ with probability $1$. More precisely, for any $\hat{u}\in S^1$ and $c>0$:
$$
\mbox{a.s. } \; \limsup_{r\to\infty} \chi_{r}(\hat{u},c) \geq 1 ~.
$$

\vskip0.2cm
The rest of the paper is organized as follows. In Section \ref{sect:modifiedFPP}, the modified Euclidean FPP model is defined. Existence and uniqueness of its geodesics are discussed. Section \ref{sect:outlineTH1} is devoted to the outline of the proof of Theorem \ref{delThm}: the Burton-Keane argument is applied to the modified FPP model. But the heart of the proof, i.e., Proposition \ref{part2Lem}, is proved in Section \ref{thm1Sec}. Finally, Section~\ref{SecProofSublin} provides a proof of Theorem~\ref{theo:sublin}.

\section{The adapted Burton-Keane argument}
\label{sect:adaptedBK}

In this section, $G(X)$ still denotes the Delaunay triangulation $\Del(X)$ and the relative neighborhood graph $\Rng(X)$ defined on the Poisson point process $X$. Our goal is to apply the Burton-Keane argument to a modified Euclidean FPP model defined as follows.

\subsection{The modified Euclidean FPP model}
\label{sect:modifiedFPP}

Let $r\in\R$ and let us consider a modified FPP model on the graph $G(X)$ in which the length of an oriented path is given by its Euclidean length, unless it crosses the vertical line $\ellv_r=\{r\}\times\R$ from right to left. In the latter case, its length is defined to be infinite. This is the reason why from now on, we have to consider oriented paths.

Let $\gamma=(x_{1},\ldots,x_{k})$ be an oriented path in $G(X)$ where $x_{1},\ldots,x_{k}$ are Poisson points. We set
$$
\nu_1^{(r)}(\gamma)=
\begin{cases}
\infty
&\text{if $\pi_1(x_{j})\ge r \ge \pi_1(x_{j+1})$ for some $j\in\{0,\ldots,k-1\}$,}\\
\nu_1(\gamma)&\text{otherwise,}
\end{cases}
$$
where $\pi_1:\R^2\to\R$ denotes the projection onto the first coordinate. As before, an oriented path $\gamma=(x_{j})_{1\leq j\leq k}$ with $k\in\mathbb{N}\cup\{\infty\}$ on $G(X)$ is called \emph{geodesic with respect to $\nu^{(r)}_1$} if
$$
\nu_1^{(r)}(\gamma) = \inf\{ \nu_1^{(r)}(\gamma') : \, \gamma' \text{ is an oriented path on $G(X)$ from $x_{1}$ to $x_{k}$} \} ~,
$$
where $\gamma[x_{i},x_{j}]$ denotes the (oriented) sub-path of $\gamma$ from $x_{i}$ to $x_{j}$. We also make the convention that an oriented path crossing the vertical line $\ellv_r$ from right to left cannot be a geodesic.

Let us remark that geodesics of the modified FPP model, that is, w.r.t. $\nu^{(r)}_1$, can be markedly different from the ones of the original FPP model. To see it, let us consider $x,x'\in X$ such that $\pi_1(x)<r<\pi_1(x')$ and let $\gamma$ be the geodesic connecting $x$ and $x'$ in the original FPP model. Then, the geodesic w.r.t. $\nu^{(r)}_1$ from $x$ to $x'$ is equal to $\gamma$ (up to the orientation) if and only if $\gamma$ crosses the vertical line $\ellv_r$ only one time. Besides, the geodesic w.r.t. $\nu^{(r)}_1$ from $x'$ to $x$ do not exist. Also, let us remark in the relative neighborhood graph $\Rng(X)$ it can happen that the only edge starting from $x$ crosses the vertical line $\ellv_r$ from right to left. In this case, there is no (finite or not) geodesic starting at $x$ in the modified FPP model. Such pathological situations do not occur in the Delaunay triangulation $\Del(X)$. Besides, in the modified FPP model on $\Del(X)$, it should be possible to prove existence of $\hat{u}$-unigeodesics for any asymptotic direction $\hat{u}$ such that $\langle \hat{u},e_{1} \rangle\geq 0$ and starting from any vertex $x$, using~\cite[Proposition 2.8]{efpp2}. However, this would still require some effort and will not be needed in the following.

Unlike existence, uniqueness of $\hat{u}$-unigeodesics -- when they exist -- in the modified FPP model is established fairly easily. Indeed, the classical argument due to Licea and Newman~\cite{lNewman} shows that if $\hat{u}\in S^1$ is chosen suitably, then for each vertex $x\in G^{}(X)$ there exists at most one semi-infinite geodesic w.r.t. $\nu^{(r)}_1$ with asymptotic direction $\hat{u}$ and starting point $x$. When it exists, it will be denoted $\gamma^{(r)}_x$.

To make the presentation self-contained, we reproduce from~\cite[Theorem 0]{lNewman} the original argument for the uniqueness of $\hat{u}$-unigeodesics. See also~\cite[Lemma 6]{efpp} for another account.
Let $D_r(\hat{u})$ be the event that for every $x\in X$ there exists at most one $\hat{u}$-unigeodesic w.r.t. $\nu_1^{(r)}$ starting from $x$. By stationarity, the probability of the event $D_r(\hat{u})$ does not depend on the value of $r\in\R$.

\begin{lemma}
\label{uniGeoLem}
It holds that $\int_{S^1}\P(D_0(\hat{u})^c)\d\hat{u}=0$ where $\d\hat{u}$ denotes the Lebesgue measure on $S^1$. In other words, for almost every $\hat{u}\in S^1$, $\P(D_0(\hat{u}))=1$.
\end{lemma}

\begin{proof}
If $D_0(\hat{u})$ does not occur, then there exists some point $x\in X$ featuring two $\hat{u}$-unigeodesics $\gamma_1$ and $\gamma_2$ that have $x$ as their only common point. Let $x_1$, $x_2$ be the respective successors of $x$ in these geodesics. Writing $\gamma_+(x)$, $\gamma_-(x)$ for the trigonometrically highest and lowest geodesics w.r.t. $\nu^{(0)}_1$ in $G^{}(X)$ starting from $x\in X$, we conclude from the planarity of $G(X)$ that at least one of $\gamma_+(x_1)$, $\gamma_+(x_2)$, $\gamma_-(x_1)$ and $\gamma_-(x_2)$ is trapped between $\gamma_1$ and $\gamma_2$. Hence, its asymptotic direction is also given by $\hat{u}$. In other words, we have $D_0(\hat{u})^c\subset H(\hat{u})$, where $H(\hat{u})$ denotes the event that there exists $x\in X$ such that $\gamma_-(x)$ or $\gamma_+(x)$ has asymptotic direction $\hat{u}$. Since in each realization, the event $H(\hat{u})$ can occur only for two directions 
of $S^1$, we deduce that $\int_{S^1}\one\{H(\hat{u})\} \d\hat{u}=0$. Hence, by Fubini's theorem,
\begin{align*}
\int_{S^1}\P(D_0(\hat{u})^c)\d\hat{u}\le\int_{S^1}\P(H(\hat{u}))\d\hat{u}=\E\int_{S^1}\one\{H(\hat{u})\}\d\hat{u}=0,
\end{align*}
as required.
\end{proof}

In the following, we choose a direction $\hat{u}_0\in S^1$ such that $\P(D_0(\hat{u}_0))=1$ and the absolute value of the angle of $\hat{u}_0$ with the $x$-axis is at most $\delta$, where $\delta>0$ is assumed to be the inverse of a sufficiently large integer, which is fixed in the entire manuscript.

\subsection{Outline of the proof of Theorem \ref{delThm}}
\label{sect:outlineTH1}

Next, we introduce an event $F_M$ to describe the existence of a distinguished geodesic in the modified FPP model that is protected from coalescing with other distinguished geodesics. To make this precise we require additional notations. Given a semi-infinite path $\gamma$ of $G(X)$ and a point $P\in\gamma$, we denote by $\gamma[P]$ the semi-infinite subpath of $\gamma$ starting at $P$. Let $H^- =(-\infty,0]\times\R$ and $H^+=[0,\infty)\times\R$ be the negative and positive vertical half-plane, respectively. We also denote by $C_\delta(P)$ the cone with apex $P\in\R^2$, asymptotic direction $\hat{u}_0$ and opening angle $\delta$.

We can now define the event $F_M$ asserting that there exist points $x_{\ms{m}}^-\in X\cap H^-$ and $x_{\ms{m}}^+\in X\cap H^+$ with the following properties:
\begin{enumerate}
\item $[x_{\ms{m}}^-,x_{\ms{m}}^+]$ forms an edge in $G(X)$ and is contained in $B_{M/2}(o)$,
\item $\gamma_{\ms{m}}=\gamma^{(0)}_{x_{\ms{m}}^+}$ exists and is contained in the dilated cone $C_\delta(Q)\oplus B_{\delta M}(o)$, where $\{Q\}=[x_{\ms{m}}^-,x_{\ms{m}}^+]\cap\ellv_0$,
\item if $z\in4\Z\times2\Z$ is such that $z\ne o$ and $\pi_1(z)\le0$, and $x_-\in X\cap(Mz+H^-)$ as well as $x_+\in X\cap(Mz+H^+)$ are such that 
\begin{enumerate}
\item $[x^-,x^+]$ forms an edge in $G(X)$ and is contained in $B_{M/2}(Mz)$,
\item $\gamma=\gamma^{(M\pi_1(z))}_{x^+}$ exists,
\end{enumerate}
then $\gamma[P]\cap\gamma_{\ms{m}}=\es$, where $P$ denotes the last intersection point of $\gamma$ and $\ellv_0$. If such $P$ does not exist, we put $\gamma[P]=\gamma$. 
\end{enumerate}
Figure~\ref{fmFig} provides an illustration of the event $F_M$. It is worth pointing out here that the geodesics involved in $F_M$ are w.r.t. $\nu_1^{(r)}$ for different values of $r$.

\begin{figure}[!ht]
\begin{center}
\psfrag{b1}{\small{$B_{M/2}(Mz)$}}
\psfrag{b2}{\small{$B_{M/2}(o)$}}
\psfrag{x1}{\small{$x^{+}$}}
\psfrag{x2}{\small{$x_{\ms{m}}^+$}}
\psfrag{u}{\small{$\hat{u}_0$}}
\psfrag{g1}{\small{$\gamma[P]$}}
\psfrag{g2}{\small{$\gamma_{\ms{m}}$}}
\psfrag{P}{\small{$P$}}
\psfrag{o}{\small{$o$}}
\includegraphics[width=14cm,height=7cm]{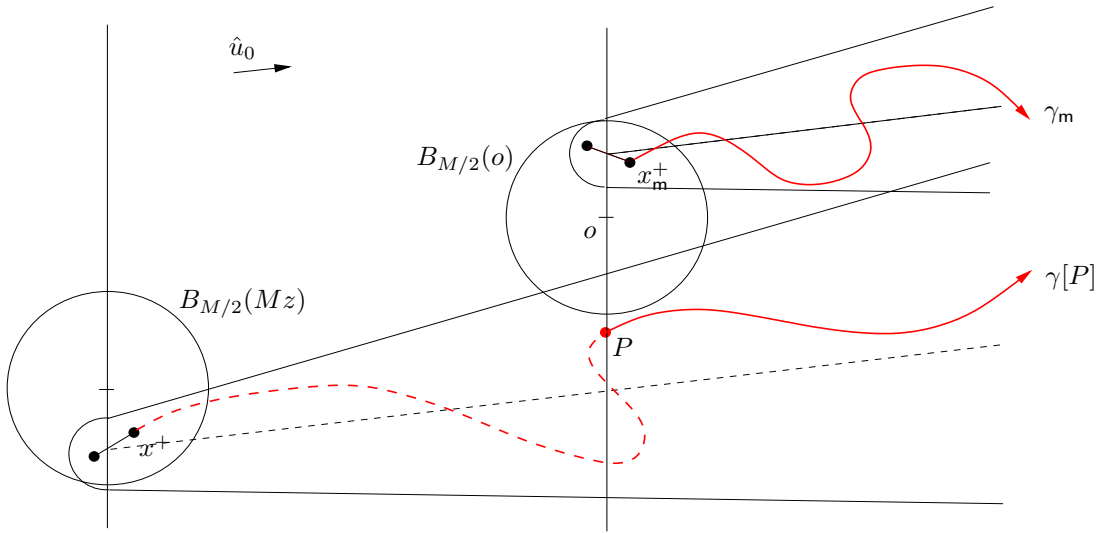}
\end{center}
\caption{\label{fmFig} This picture represents the event $F_M$. The solid red curves are the geodesics $\gamma_{\ms{m}}=\gamma^{(0)}_{x_{\ms{m}}^+}$ and $\gamma[P]$ where $\gamma=\gamma^{(M\pi_1(z))}_{x^+}$. They are both contained in the corresponding dilated cones and they do not overlap. Let us remark that for clarity, this picture does not respect the fact that $z$ belongs to $4\Z\times2\Z$.}
\end{figure}



We now consider the family of $\hat{u}_0$-unigeodesics $\{\gamma_x\}_{x\in X}$ in the original FPP model on $G(X)$. This family can be used to define a forest $\mc{F}=\mc{F}_{\hat{u}_0}$ with vertex set $X$ by drawing an edge from $x$ to $y$ if $[x,y]$ is the first edge in the geodesic $\gamma_x$. If $N$ denotes the number of connected components in this forest, then Theorem~\ref{delThm} is equivalent to the assertion that $\P(N\ge2)=0$. Of course, by isotropy, we could choose $\hat{u}_0=e_1$, but for our argument it will be notationally convenient to have a certain flexibility in the choice of the direction. We will always assume that the angle between $\hat{u}_0$ and $e_1$ is at most $\delta$.

The heart of our paper is to show that the event $F_M$ occurs with positive probability if the event $\{N\ge2\}$ does so. Its proof is devoted to Section \ref{thm1Sec}.

\begin{proposition}
\label{part2Lem}
If $\P(N\ge2)>0$, then $\P(F_M)>0$ for $M$ large enough.
\end{proposition}

The final part of the proof of Theorem~\ref{delThm} is to prove that $\P(F_M)=0$. In contrast to Proposition~\ref{part2Lem}, the classical argument~\cite[Theorem 1]{lNewman} applies without substantial further complications.

\begin{proposition}
\label{part3Lem}
It holds that $\P(F_{M})=0$, for $M$ large enough.
\end{proposition}

\begin{proof}[Proof of Proposition~\ref{part3Lem}]
By Lemma~\ref{part2Lem}, we may fix $M\ge1$ such that $\P(F_M)>0$.
For $L\ge1$ let $f(L)$ denote the number of $z\in(4\Z\cap[-4L,4L])\times(2\Z\cap[-2L,2L])$ such that $(X-Mz)\in F_M$. Then, by stationarity of $G(X)$ and the choice of $M$,  the first moment of $f(L)$ grows quadratically in $L$. On the other hand, let $f'(L)$ denote the number of edges in $G(X)$ intersecting the boundary of $[-8LM,8LM]\times[-4LM,4LM]$. Then $\E f'(L)$ grows linearly in $L$. Hence, it suffices to show that if $L$ is sufficiently large, then $f(L)\le f'(L)$ holds almost surely.

In order to prove this claim, let $z,z'\in(4\Z\cap[-4L,4L])\times(2\Z\cap[-2L,2L])$ be such that $z\ne z'$, $(X-Mz)\in F_M$ and $(X-Mz')\in F_M$. Furthermore, let $\gamma_{z,\ms{m}}=\gamma^{(M\pi_1(z))}_{x^+_{z,\ms{m}}}$ and $\gamma_{z',\ms{m}}=\gamma^{(M\pi_1(z'))}_{x^+_{z',\ms{m}}}$ denote the geodesics that are guaranteed by the events $(X-Mz)\in F_M$ and $(X-Mz')\in F_M$. Then, we claim that $\gamma_{z,\ms{m}}$ and $\gamma_{z',\ms{m}}$ do not leave the rectangle $[-8LM,8LM]\times[-4LM,4LM]$ via the same edge. Indeed, without loss of generality, we may assume that $\pi_1(z)\le \pi_1(z')$. If $\pi_1(z)=\pi_1(z')$, then we know from the definition of $F_M$ that $\gamma_{z,\ms{m}}$ and $\gamma_{z',\ms{m}}$ are disjoint. Hence, we may assume that $\pi_1(z)<\pi_1(z')$. Since $X-Mz\in F_M$, the geodesic $\gamma_{z,\ms{m}}$ is contained in the cone $C_\delta(Q)\oplus B_{\delta M}(o)$, where $\{Q\}=[x^-_{z,\ms{m}},x^+_{z,\ms{m}}]\cap \ellv_{Mz}$. Hence, we conclude that the last intersection point $P$ of $\gamma_{z,\ms{m}}$ with the vertical line $\ellv_{Mz'}$ is contained in $[-8LM,8LM]\times[-4LM,4LM]$. Now, the occurrence of the event $F_M$ implies that $\gamma_{z,\ms{m}}[P]$ and $\gamma_{z',\ms{m}}$ are disjoint, so that they leave the rectangle $[-8LM,8LM]\times[-4LM,4LM]$ via different edges.
\end{proof}

\section{Proof of Proposition~\ref{part2Lem}}
\label{thm1Sec}

\input{lem2-7.tex}

\section{Proof of Theorem \ref{theo:sublin}}
\label{SecProofSublin}

The proof of Theorem \ref{theo:sublin} follows the method developed in~\cite{sublin}. In order to make the paper self-contained, we recall the main steps of the method and insist on the new parts which are proper to the context $G(X)=\Del(X)$ or $\Rng(X)$.

To begin with, we provide an overview for the proof of Theorem~\ref{theo:sublin}. Let $\hat{u}\in S^1$ be a given direction. By isotropy, it suffices to prove that the expectation of $\chi_{r}(\hat{u},2\pi)$ tends to $0$ as $r$ tends to infinity. The proof works as well when $2\pi$ is replaced with any $c>0$.

The spirit of the proof of Theorem \ref{theo:sublin} is the following. First, a uniform moment condition reduces the proof to the convergence in probability of $\chi_{r}(\hat{u},2\pi)$ to $0$. Thus, far away from the origin, the radial character of the geodesics of tree $\T_{x^{\star}}$ vanishes. In other words, when $r$ is large, with high probability, the shortest-path tree $\T_{x^{\star}}$ locally looks like to a directed forest around $r\hat{u}$. This directed forest is in fact made up of all the semi-infinite geodesics $\gamma_{x}$ with direction $-\hat{u}$ starting from all the Poisson points $x\in X$. Let us denote this forest by $\F_{-\hat{u}}$. Henceforth, to a semi-infinite geodesic in $\T_{x^{\star}}$ crossing the arc of $S_r(o)$ centered at $r\hat{u}$ and with length $2\pi$, it corresponds a bi-infinite geodesic in the directed forest $\F_{-\hat{u}}$. Now, Theorem~\ref{delThm} states that this event should not occur.\\

\textbf{STEP 1:} A classical use of Fubini's theorem-- see \cite[Lemma 6]{efpp}, \cite[Theorem 0]{lNewman} or Lemma \ref{uniGeoLem} above --ensures the existence with probability $1$ of exactly one semi-infinite geodesic, say $\gamma_{x}$, with direction $-\hat{u}$ starting from every Poisson point $x\in X$. The collection of these semi-infinite geodesics provides a directed forest $\F_{-\hat{u}}$. The goal of this first step is to approximate the shortest-path tree $\T_{x^{\star}}$ inside the disk $B_L(r\hat{u})$ by the directed forest $\F_{-\hat{u}}$: see (\ref{approxDF}) below.

Let us remark that the graphs $\T_{x^{\star}}$ and $\F_{-\hat{u}}$ can be built simultaneously on the same vertex set $X$. Except for the root $x^{\star}$ in $\T_{x^{\star}}$, all their vertices are with outdegree $1$. A measurable function $F$ is said \textit{local} if there exists a deterministic real number $L>0$ such that, for any $z\in\R^{2}$, the quantities $F(z,\T_{x^{\star}})$ and $F(z,\F_{-\hat{u}})$ are equal whenever each Poisson point $x\in X\cap B_L(z)$ admits the same outgoing vertex in $\T_{x^{\star}}$ and $\F_{-\hat{u}}$. Our local approximation result is written through the use of local functions: for any local function $F$,
\begin{equation}
\label{approxDF}
\lim_{r\to\infty} d_{\ms{TV}} \Big(F(r\hat{u},\T_{x^{\star}}),F(r\hat{u},\F_{-\hat{u}}) \Big) = 0 ~,
\end{equation}
where $d_{\ms{TV}}$ denotes the total variation distance.

Let $F$ be a local function with local parameter $L$. Then, by translation invariance of the directed forest $\F_{-\hat{u}}$,
\begin{eqnarray}
d_{\ms{TV}} \left( F(r\hat{u},\T_{x^{\star}}) , F(r\hat{u},\F_{-\hat{u}}) \right) & = & d_{\ms{TV}} \left( F(r\hat{u},\T_{x^{\star}}) , F(r\hat{u},\F_{-\hat{u}}) \right) \nonumber\\
& \leq & \P \left( F(r\hat{u},\T_{x^{\star}}) \not= F(r\hat{u},\F_{-\hat{u}}) \right) ~.
\end{eqnarray}
The event $F(r\hat{u},\T_{x^{\star}})\not= F(r\hat{u},\F_{-\hat{u}})$ implies the existence of a Poisson point $x$ in $B_L(r\hat{u})$ whose outgoing vertices in $\T_{x^{\star}}$ and $\F_{-\hat{u}}$ are different. So, by uniqueness, the geodesics $\gamma_{x^{\star},x}$ (in $\T_{x^{\star}}$) and $\gamma_{x}$ (in $\F_{-\hat{u}}$) have only $x$ in common. Hence, for $\eps\in(0;1)$,
$$
d_{\ms{TV}} \left( F(r\hat{u},\T_{x^{\star}}) , F(r\hat{u},\F_{-\hat{u}}) \right) \leq \P \left(\begin{array}{c}
\exists x \in X\cap B_L(r\hat{u}) \; \mbox{ such that }\\
\gamma_{x^{\star},x}\cap\gamma_{x} = \{x\} \; \mbox{ and } x^{\star}\in B_{r^\eps}(o)
\end{array}\right) \; + o(1) ~.
$$
Thus, because of translation invariance and the identity $\gamma_{x^{\star},x}=\gamma_{x,x^{\star}}$, we can bound the term $d_{\ms{TV}}( F(r\hat{u},\T_{x^{\star}}),F(r\hat{u},\F_{-\hat{u}}))$ by
\begin{equation}
\label{step1-1}
\P \left(\begin{array}{c}
\exists x \in X\cap B_L(o) , \, \exists x' \in X\cap B_{r^\eps}(-r\hat{u}) \\
\mbox{such that } \; \gamma_{x,x'}\cap\gamma_{x} = \{x\}
\end{array}\right) \; + o(1) ~.
\end{equation}

\begin{figure}[!ht]
\begin{center}
\psfrag{b1}{\small{$B_L(o)$}}
\psfrag{b}{\small{$B_{r^\eps}(-r\hat{u})$}}
\psfrag{g1}{\small{$\gamma_{x,x'}$}}
\psfrag{g2}{\small{$\gamma_{x}$}}
\psfrag{x1}{\small{$x$}}
\psfrag{x2}{\small{$x'$}}
\includegraphics[width=13cm,height=3.5cm]{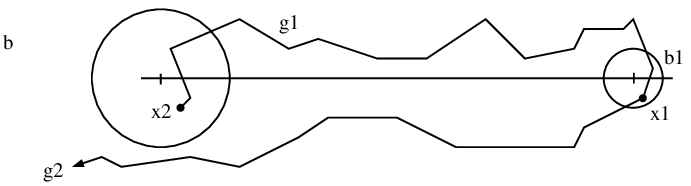}
\end{center}
\caption{\label{fig:SublinDel} This picture represents the event appearing in (\ref{step1-1}) with $\hat{u}=(1,0)$. Poisson points $x$ and $x'$ respectively belong to the disks $B_L(o)$ and $B_{r^\eps}(-r\hat{u})$. The shortest-path tree $\T_{x}$ contains a semi-infinite geodesic $\gamma_{x}$ with asymptotic direction $-\hat{u}$ and a geodesic $\gamma_{x,x'}$ whose endpoint is in $B_{r^\eps}(-r\hat{u})$. Furthermore, $\gamma_{x}$ and $\gamma_{x,x'}$ have only the vertex $x$ in common.}
\end{figure}

Besides, we can require that the geodesics $\gamma_{x}$ and $\gamma_{x,x'}$ belong to a thin cone with direction $-r\hat{u}$. For $\eta>0$ and $z\in\R^{2}$, let $C(z,\eta)$ be the cone with apex $o$, direction the vector $z$ and opening angle $\eta$: $C(z,\eta)=\{z'\in\R^{2}, \theta(z,z')\leq\eta\}$ where $\theta(z,z')$ is the absolute value of the angle (in $[0;\pi]$) between vectors $z$ and $z'$. On the one hand, the semi-infinite geodesic $\gamma_{x}$ has asymptotic direction $-\hat{u}$. So, with high probability, for any $\eta>0$ and $M$ large enough, its restriction to the outside of the disk $B_M(o)$ is included in the cone $C(-\hat{u},\eta)$. See \cite[Lemma 13]{sublin} for details. On the other hand, the same statement holds for the geodesic $\gamma_{x,x'}$ using the moderate deviations result \cite[Theorem 2]{modDev}. See \cite[Lemma 14]{sublin} for details. Consequently, for any $\eta>0$ and $M,r$ large enough, the term $d_{\ms{TV}}(F(r\hat{u},\T_{x^{\star}}),F(o,\F_{-\hat{u}}))$ is bounded by 
\begin{equation}
\label{step1-2}
\P \left(\begin{array}{c}
\exists x \in X\cap B_L(o) , \, \exists x' \in X\cap B_{r^\eps}(-r\hat{u})\\
\mbox{such that } \; \gamma_{x,x'}\cap\gamma_{x} = \{x\} \; \mbox{ and}\\
(\gamma_{x,x'}\cup\gamma_{x})\cap B_M(o)^{c} \subset C(-\hat{u},\eta)
\end{array}\right) \; + o(1) ~.
\end{equation}
On the event described in (\ref{step1-2}), the shortest-path tree $\T_{x}$ contains two disjoint (except the root $x$) geodesics with direction $-\hat{u}$ which are as long as we want. However, with probability $1$, $\T_{x}$ contains at most one semi-infinite geodesic with (deterministic) asymptotic direction $-\hat{u}$. So (\ref{step1-2}) is a $o(1)$ which leads to (\ref{approxDF}).\\

\textbf{STEP 2:} The goal of this second step is to state that the directed forest $\F_{-\hat{u}}$ does not contain any bi-infinite geodesic with probability $1$. This is a consequence of coalescence of semi-infinite geodesics (i.e. Theorem \ref{delThm}).

Let $\hat{v}\in S^{1}$ be orthogonal to $\hat{u}$. Let $\ell$ be the line spanned by the vector $\hat{v}$ and, for any $m>0$, let $\ell_{m}=\ell-m\hat{u}$. For $x<y$, we also denote by $\ell_m(x,y)$ the subset of $\ell_{m}$ defined by
$$
\ell_m(x,y) = \{-m\hat{u} + b\hat{v} \in \R^{2} ; x\leq b<y \} ~.
$$
Thus, we denote by $K[\ell_{0}(x,y)]$ the number of elements $P\in\ell_{0}(x,y)$ which are defined as the last intersection point between a bi-infinite geodesic of the directed forest $\F_{-\hat{u}}$ and the line $\ell_{0}$. In the same way, we denote by $K[\ell_{0}(x,y), \ell_m]$ the number of elements $P\in\ell_{m}$ which are defined as the last intersection point between a bi-infinite geodesic $\gamma$ of $\F_{-\hat{u}}$ and the line $\ell_{m}$, and whose the last intersection point between $\gamma$ and $\ell_{0}$ belongs to $\ell_{0}(x,y)$.

\begin{figure}[!ht]
\begin{center}
\psfrag{a}{\small{$-\hat{u}$}}
\psfrag{b}{\small{$\ell_m$}}
\psfrag{c}{\small{$\ell_0$}}
\includegraphics[width=12cm,height=7cm]{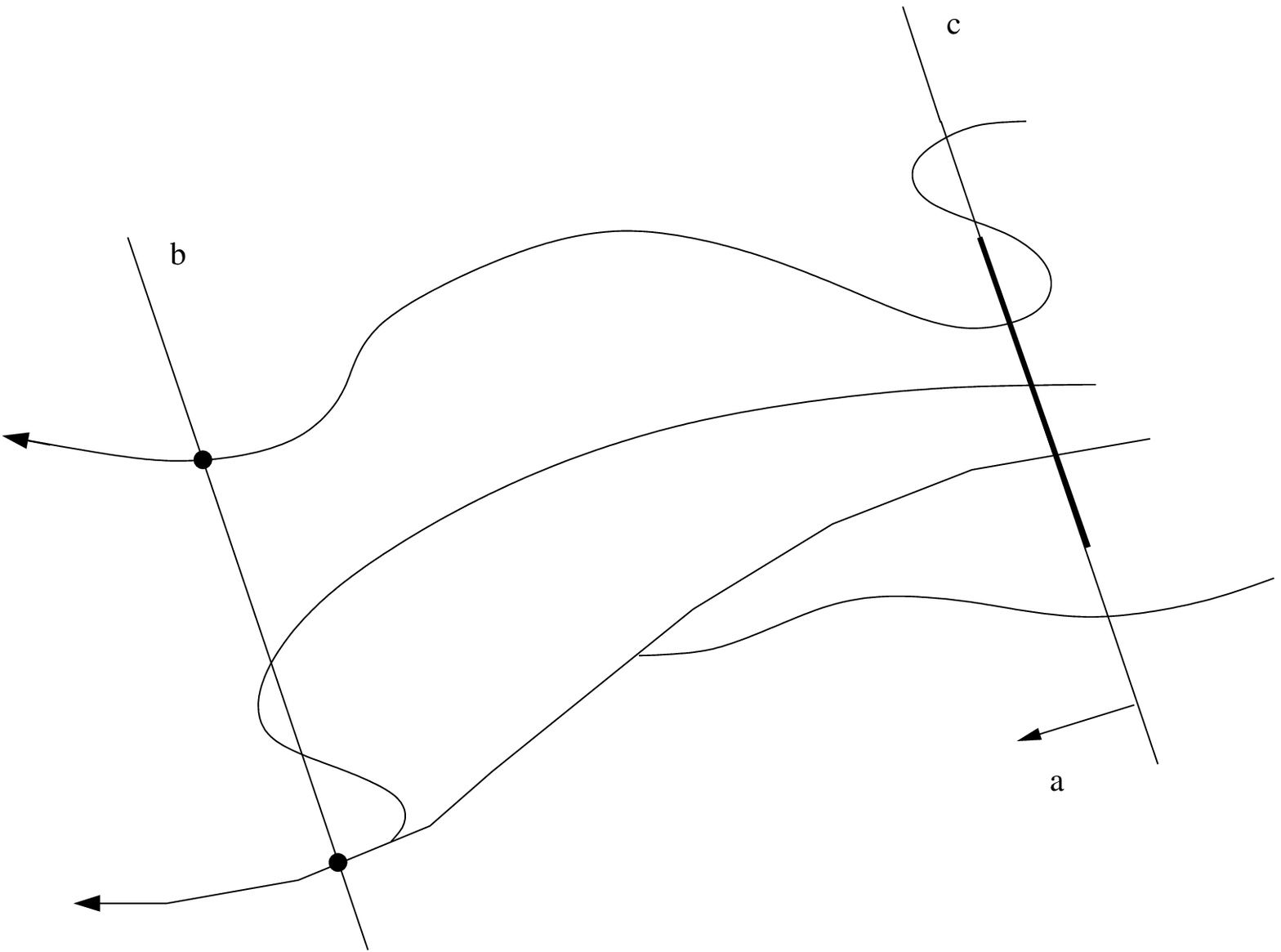}
\end{center}
\caption{\label{fig:Step3} The thick segment on the line $\ell_{0}$ represents $\ell_{0}(x,y)$. The two black points on $\ell_{m}$ are elements counted by $K[\ell_{0}(x,y), \ell_m]$. On this picture, $K[\ell_{0}(x,y), \ell_m]$ equals $2$ whereas $K[\ell_{0}(x,y)]$ equals $3$: two bi-infinite geodesics counted by $K[\ell_{0}(x,y)]$ merge before the line $\ell_{m}$.}
\end{figure}

However only the inequality $K[\ell_{0}(0,L)]\geq K[\ell_{0}(0,L),\ell_m]$ holds a.s. (see Figure \ref{fig:Step3}), by stationarity of the directed forest $\F_{-\hat{u}}$, it is possible to prove the identity
\begin{equation}
\label{step2-1}
\E K[\ell_{0}(0,L)] = \E K[\ell_{0}(0,L),\ell_m]
\end{equation}
for any integers $L,m>0$. See \cite[Section 6]{sublin} for details. Now, thanks to the coalescence of the semi-infinite geodesics of $\F_{-\hat{u}}$, the non-increasing sequence $(K[\ell_{0}(0,L),\ell_m])_{m\geq0}$ a.s. converges to a limit smaller than $1$. By the Lebesgue's Dominated Convergence Theorem, we get for any integer $L>0$:
$$
L \E K[\ell_{0}(0,1)] = \E K[\ell_{0}(0,L)] = \lim_{m\to\infty} \E K[\ell_{0}(0,L),\ell_m] \leq 1 ~.
$$
This forces $\E K[\ell_{0}(0,1)]$ to be null. So, a.s. there is no bi-infinite geodesic crossing the vertical segment $\ell_{0}(0,1)$. We conclude by stationarity.\\

\textbf{STEP 3:} This third step consists in combining the results of the two previous ones in order to establish the convergence in probability of $\chi_{r}(\hat{u},2\pi)$ to $0$.

Let $a(r)$ be the arc of the circle $S_r(o)$ centered at $r\hat{u}$ and with length $2\pi$. When $r$ is large, $a(r)$ becomes close (for the Hausdorff distance) to the segment $I(r)$ centered at $r\hat{u}$ and with length $2\pi$ which is orthogonal to $\hat{u}$. With high probability, the event $\chi_{r}(\hat{u},2\pi)\geq 1$ implies the existence of a geodesic in the shortest-path tree $\T_{x^{\star}}$ crossing $I(r)$ and exiting the disk $B_R(r\hat{u})$, for any $R$-- when one walks away from the root $x^{\star}$. Precisely, for any $R>0$, $\P(\chi_{r}(\hat{u},2\pi)\geq 1)$ is smaller than
\begin{equation}
\label{step3-1}
\P \left(\begin{array}{c}
\exists x_{1},x_{2},x_{3} \in X \; \mbox{ such that } \\
\gamma_{x_{1},x_{2}}\cap I(r) \not= \emptyset \, , \; \gamma_{x_{1},x_{2}} \subset B_R(r\hat{u}) \, ,\\
\gamma_{x_{1},x_{3}}=\gamma_{x_{1},x_{2}} \cup \{(x_{2},x_{3})\} \, , \; x_{3} \in B_{2R}(r\hat{u}) \setminus B_R(r\hat{u})
\end{array}\right) + o(1) ~,
\end{equation}
as $r$ tends to infinity. This latter inequality relies on the fact that, with high probability when $R$ tends to infinity, the graph $G(X)$ contains no edge whose endpoints respectively belong to $B_R(r\hat{u})$ and the outside of $B_{2R}(r\hat{u})$. In both cases $G(X)=\Rng(X)$ or $G(X)=\Del(X)$, this would imply the existence of a random disk avoiding the Poisson point process $X$, overlapping $B_R(r\hat{u})$ and with diameter larger than $R$. Let $\kappa=\lfloor6\pi\rfloor+1$. There exists a deterministic sequence $u_{1},\ldots,u_{\kappa}$ of points of the circle $S_{3R/2}(r\hat{u})$ such that $|u_{i}-u_{i+1}|\leq R/2$ for $i=1,\ldots,\kappa$ (where the index $i$ is taken modulo $\kappa$). Then, at least one of the deterministic disks $B_{R/2}(u_{i})$ would avoid the Poisson point process $X$. Such an event should not occur with high probability as $R$ tend to infinity (uniformly in $r$).

Since the geodesic $\gamma_{x_{1},x_{3}}$ is included in $B_{2R}(r\hat{u})$ then the event written in \eqref{step3-1} can be described using a local function with local parameter $2R$. So, we can apply the result of Step 1. Let $I(0)$ be the segment centered at the origin $o$ and with length $2\pi$ which is orthogonal to $\hat{u}$. For any $R>0$, the probability $\P(\chi_{r}(\hat{u},2\pi)\geq 1)$ is bounded by
\begin{equation}
\label{step3-2}
\P \left(\begin{array}{c}
\exists x \in X \; \mbox{ with } \; x \notin B_R(o)\\
\mbox{whose semi-infinite geodesic $\gamma_{x}$ in $\F_{-\hat{u}}$}\\
\mbox{crosses the segment $I(0)$}.
\end{array}\right) + o(1) ~,
\end{equation}
as $r$ tends to infinity. Now, thanks to Step 2, there is no bi-infinite geodesic in the directed forest $\F_{-\hat{u}}$ with probability $1$. So we can choose the radius $R$ large enough so that the probability in (\ref{step3-2}) is as small as we want.\\

\textbf{STEP 4:} It then remains to state the following uniform moment condition to strengthen the convergence of $\chi_{r}(\hat{u},2\pi)$ to $0$ in the $L^1$ sense:
\begin{equation}
\label{UnifMomCond}
\exists r_{0} > 0 , \; \sup_{r\geq r_{0}} \E \chi_{r}(\hat{u},2\pi)^{2} \, < \, \infty ~.
\end{equation}
See \cite[Section 3.1]{sublin} for details. Let $\psi_{r}$ be the number of edges of the shortest-path tree $\T_{x^{\star}}$ crossing the arc $a(r)$. Of course, $\chi_{r}(\hat{u},2\pi)$ is smaller than $\psi_{r}$ and it suffices-- to get (\ref{UnifMomCond}) --to prove that $\P(\psi_{r}>n)$ decreases fast enough and uniformly on $r$.

First, the number of Poisson points inside the disk $B_R(r\hat{u})$ is controlled by \cite[Lemma 11.1.1]{Talagrand}: for any $n,r,R$,
\begin{equation}
\label{InegTalagrand}
\P ( \#(X \cap B_R(r\hat{u})) > n ) \leq e^{-n \ln \left( \frac{n}{e\pi R^{2}} \right)} ~.
\end{equation}
[Note that (\ref{InegTalagrand}) can also be obtained from Stirling's formula when $R$ and $n$ are large.] From now on, we assume that there are no more than $n$ Poisson points inside $B_R(r\hat{u})$. Since, the tree $\T_{x^{\star}}$ admits more than $n$ edges crossing the arc $a(r)$, i.e. $\psi_{r}>n$, then necessarily one of these edges has (at least) one of its endpoints outside $B_R(r\hat{u})$. In both cases $G(X)=\Rng(X)$ or $G(X)=\Del(X)$, this implies the existence of a random disk avoiding the Poisson point process $X$, overlapping the arc $a(r)$ and with diameter larger than $R-2\pi$. It is not difficult to prove that the probability of such an event decreases exponentially fast with $R$ and uniformly on $r$ and $\hat{u}$. That is to say there exist positive constants $c,c'$ such that for $R$ large enough,
\begin{equation}
\label{Inegstep4}
\P ( \#(X \cap B_R(r\hat{u})) \leq n \, , \; \psi_{r} > n ) \leq c e^{-c'R} ~.
\end{equation}
The searched result follows from (\ref{InegTalagrand}) and (\ref{Inegstep4}) by taking for example $R=n^{1/4}$.

{\small \providecommand{\noopsort}[1]{}
}

\end{document}

%% file: lem2-7.tex
Since the proof of Proposition~\ref{part2Lem} is rather long, we provide the reader with a brief outlook. We are going to define three events $E_{M}$, $A'_{M}$ and $A''_{M}$ such that their intersection implies $F_M$ and occurs with positive probability, therefore proving Proposition~\ref{part2Lem}.

The event $E_{M}$ mainly ensures the existence of three disjoint geodesics $\gamma_{\ms{u}}$, $\gamma_{\ms{m}}$ and $\gamma_{\ms{d}}$, all starting from the segment $\{0\}\times[-\delta M,\delta M]$ and included in $H^{+}$. As in the classic Burton-Keane argument, the role of $\gamma_{\ms{u}}$ and $\gamma_{\ms{d}}$ is to protect $\gamma_{\ms{m}}$ from above and below, respectively. In Lemma~\ref{geodConstrLem}, it is proved that the event $\{X\in E_M\}$ occurs with high probability whenever $\P(N\ge2)>0$.

To protect $\gamma_{\ms{m}}$ from the left, we turn the $(4M\times 2M)$-box $R_{M}=[-4M,0]\times[-M,M]$ into an obstacle. To do it, we first assume that the Poisson point process $X$ does not contain any points in the $(3M\times 2M)$-box $R_M^-=[-4M,-M]\times[-M,M]$. This results in a structural change of the Delaunay triangulation inside $R_M^-$, which is illustrated in Figure~\ref{deleteFig}. To ensure that the sketch in Figure~\ref{deleteFig} is accurate, we need to make some assumptions on the configuration of $X$ outside $R_M^-$. This is the role of the event $A'_{M}$. Let us set $X^{(1)}=X\setminus R_{M}^-$. Lemma~\ref{highProbLem1} says that $\P(X^{(1)}\in A_M')$ tends to $1$ as $M\to\infty$. Moreover, under the event $\{X^{(1)}\in A'_M\}$, the events $\{X^{(1)}\in E_M\}$ and $\{X\in E_M\}$ are equivalent.

\begin{figure}[!htpb]
\centering
\begin{tikzpicture}[scale=1]
\draw (-4,-1)--(0,-1)--(0,1)--(-4,1)--(-4,-1);
\fill (0,0) circle (2pt);
\coordinate[label=-0:$o$] (B) at (-0,0);
\coordinate[label=-90:$M$] (B) at (-0.5,-1.1);
\coordinate[label=-90:$3M$] (B) at (-2.5,-1.1);
\draw[decorate,decoration={brace,mirror}] (-4,-1.1) -- (-1,-1.1);
\draw[decorate,decoration={brace,mirror}] (-1,-1.1) -- (0,-1.1);
\draw[dashed] (-1,1)--(-1,-1);
\draw[red] (-3,1)--(-3,-1);
\draw[red] (-2.1,1)--(-2.1,-1);
\draw[red] (-2.2,1)--(-2.2,-1);
\draw[red] (-2.3,1)--(-2.3,-1);
\draw[red] (-2.4,1)--(-2.4,-1);
\draw[red] (-2.5,1)--(-2.5,-1);
\draw[red] (-2.6,1)--(-2.6,-1);
\draw[red] (-2.7,1)--(-2.7,-1);
\draw[red] (-2.8,1)--(-2.8,-1);
\draw[red] (-2.9,1)--(-2.9,-1);
\draw[red] (-2,1)--(-2,-1);
\draw[red] (-3,-1)--(-4,0);
\draw[red] (-3.1,-1)--(-4,-0.10);
\draw[red] (-3.2,-1)--(-4,-0.20);
\draw[red] (-3.3,-1)--(-4,-0.30);
\draw[red] (-3.4,-1)--(-4,-0.40);
\draw[red] (-3.5,-1)--(-4,-0.50);
\draw[red] (-3.6,-1)--(-4,-0.60);
\draw[red] (-3.7,-1)--(-4,-0.70);
\draw[red] (-3.8,-1)--(-4,-0.80);
\draw[red] (-3.9,-1)--(-4,-0.90);
\draw[red] (-3.1,1)--(-4,0.10);
\draw[red] (-3.2,1)--(-4,0.20);
\draw[red] (-3.3,1)--(-4,0.30);
\draw[red] (-3.4,1)--(-4,0.40);
\draw[red] (-3.5,1)--(-4,0.50);
\draw[red] (-3.6,1)--(-4,0.60);
\draw[red] (-3.7,1)--(-4,0.70);
\draw[red] (-3.8,1)--(-4,0.80);
\draw[red] (-3.9,1)--(-4,0.90);
\draw[red] (-1.1,1)--(-1,0.90);
\draw[red] (-1.2,1)--(-1,0.80);
\draw[red] (-1.3,1)--(-1,0.70);
\draw[red] (-1.4,1)--(-1,0.60);
\draw[red] (-1.5,1)--(-1,0.50);
\draw[red] (-1.6,1)--(-1,0.40);
\draw[red] (-1.7,1)--(-1,0.30);
\draw[red] (-1.8,1)--(-1,0.20);
\draw[red] (-1.9,1)--(-1,0.10);
\draw[red] (-1.1,-1)--(-1,-0.90);
\draw[red] (-1.2,-1)--(-1,-0.80);
\draw[red] (-1.3,-1)--(-1,-0.70);
\draw[red] (-1.4,-1)--(-1,-0.60);
\draw[red] (-1.5,-1)--(-1,-0.50);
\draw[red] (-1.6,-1)--(-1,-0.40);
\draw[red] (-1.7,-1)--(-1,-0.30);
\draw[red] (-1.8,-1)--(-1,-0.20);
\draw[red] (-1.9,-1)--(-1,-0.10);
\end{tikzpicture}
\caption{Edges in the Delaunay triangulation in the event $\{X\cap R_M^-=\es\}\cap \{X^{(1)}\in A'_M\}$.}
\label{deleteFig}
\end{figure}

Third, we need a certain control over lengths of geodesics in $G(X)$. This property will be encoded in an event $A''_M$ that only depends on the configuration of $X$ in $\R^2\setminus R_M^-$. We will see in Lemma~\ref{highProbLem2} that $\P(X^{(1)}\in A_M'')$ tends to $1$ as $M\to\infty$.

Finally, we will show in Lemma~\ref{fImplLem} that $X^{(1)}\in E_{M}\cap A'_{M}\cap A''_{M}$ almost surely implies that $X^{(1)}\in F_{M}$. Hence, the following computation proves Proposition~\ref{part2Lem}:
\begin{align*}
0&<\P(X^{(1)}\in E_{M}\cap A'_{M}\cap A''_{M}) \P(X\cap R^-_{M}=\emptyset)\\
&= \P(X^{(1)}\in E_{M}\cap A'_{M}\cap A''_{M}, X\cap R^-_{M}=\emptyset)\\
&\leq \P(X^{(1)}\in F_{M}, X\cap R_M^-=\emptyset)\\
&\leq \P(X \in F_{M}).
\end{align*}


\subsection{The event $E_{M}$}

In this section, using the assumption that $\P(N\ge2)>0$, we explain how to protect the geodesic $\gamma_{\ms{m}}$ from above and below by two semi-infinite geodesics $\gamma_{\ms{u}}$ and $\gamma_{\ms{d}}$ which do not coalesce with $\gamma_{\ms{m}}$.

As the main result of this section, we can construct the desired distinguished geodesics. For $M\ge1$, we let 
$E_{M}$ denote the event that there exist 
$$P_{\ms{u}},P_{\ms{m}},P_{\ms{d}}\in G(X)\cap(\{0\}\times [-\delta M,\delta M]),$$
such that 
\begin{enumerate}
\item $\pi_2(P_{\ms{u}})>\pi_2(P_{\ms{m}})>\pi_2(P_{\ms{d}})$, and 
\item $\gamma_{\ms{u}}$, $\gamma_{\ms{m}}$ and $\gamma_{\ms{d}}$ exist and are disjoint, where we put $\gamma_{\ms{u}}=\gamma^{(0)}_{x_{\ms{u}}^+}$, $\gamma_{\ms{m}}=\gamma^{(0)}_{x_{\ms{m}}^+}$ and $\gamma_{\ms{d}}=\gamma^{(0)}_{x_{\ms{d}}^+}$.
\end{enumerate}
Now, we show that the event $E_{M}$ occurs with high probability.

\begin{lemma}
\label{geodConstrLem}
If $\P(N\ge2)>0$, then $\lim_{M\to\infty}\P(E_{M})=1$.
\end{lemma}

\begin{proof}

As a first step, we consider geodesics that do not backtrack behind the vertical line $\ellv_0$. More precisely, we define the linear point process $Y$ to consist of those $P\in l^{\ms{v}}_0$ that can be represented as $P=[x^-,x^+]\cap\ellv_0$, where $x^-,x^+\in X$ are assumed to be connected by an edge in $G(X)$. Moreover, we put $\gamma_P=\gamma_{x^+}$, where we assume that $x^+$ is chosen such that $\pi_1(x^+)>0$. Then, we let $Y'\subset Y$ denote the thinning of $Y$ consisting of all $P\in Y$ for which the geodesic $\gamma_P$ is contained in the positive half-plane $H^+$. In the following, it will play an important role that if $\gamma_P$ is a geodesic in $G(X)$ with respect to $\nu_1$ and $P\in Y'$, then $\gamma_P$ is also a geodesic with respect to $\nu_1^{(0)}$, i.e., $\gamma_P=\gamma_P^{(0)}$. Identifying $\ellv_0$ with the real line, we think of $Y'$ as one-dimensional stationary and ergodic point process. Since $\gamma_P$ has asymptotic direction $\hat{u}_0$, the intensity of $Y'$ is positive when regarded as one-dimensional point process.

In the following, it will be important to exert some control over the amount of fluctuation of the distinguished geodesics. More precisely, the definition of asymptotic direction shows that for every $P=[x^-,x^+]\cap \ellv_0\in Y'$ there exists a random $k>0$ such that $|x^--x^+|\le k$ and $\gamma_P\subset C_\delta(P)\oplus B_k(o)$. 
In particular, there exists a deterministic $K_0>0$ such that the thinning $Y''$ of $Y'$ consisting of all $P=[x^-,x^+]\cap \ellv_0\in Y'$ such that $|x^--x^+|\le K_0$ and $\gamma_P$ is contained in the dilated cone $C_\delta(P)\oplus B_{K_0}(o)$ forms a stationary and ergodic point process with positive intensity. The retention condition for $Y''$ is illustrated in Figure~\ref{k0Fig}.

\begin{figure}[!htpb]
\centering
\begin{tikzpicture}[rotate=-8,scale=1.0]
\fill (0,0) circle (1pt);
\draw[dashed] (0,0)--(5,1);
\draw[dashed] (-0.1,0.5)--(4.8,2);
\draw[dotted] (-0.1,0.5)--(4.9,1.5);
\draw[dashed] (0.1,-0.5)--(4.8,0);
\draw[dashed] (-0.1,0.5) arc (109:259:0.51);
\draw[red,->] (4.31,1.4) arc (11.3:16.3:4.59);
\draw[decorate,decoration={brace,mirror}] (-0.15,0.47) -- (-0.05,-0.03);

\draw (0,0) .. controls (2,-1) and (3.6,2) .. (5,1);
\coordinate[label=-90:$P$] (u) at (0,0);
\coordinate[label=-90:$\gamma_P$] (u) at (2.1,0.2);
\coordinate[label=-180:$K_0$] (u) at (-0.1,0.2);
\coordinate[label=0:\textcolor{red}{$\delta$}] (u) at (3.8,1.55);

\end{tikzpicture}
\caption{Illustration of the retention condition for $Y''$}
\label{k0Fig}
\end{figure}
First, we claim that for every $t\in\R$ the probability that the geodesics $\gamma_P$ coalesce for all $P\in Y''\cap(\{0\}\times[t,\infty))$ is equal to $0$. Indeed, otherwise stationarity and planarity would imply that $\P(N\ge2)=0$. In particular, 
$$\lim_{M\to\infty}\P(E_M')=1,$$
where $E_{M}'$ denotes the event that there exist $P,P',P''\in Y''\cap(\{0\}\times\R)$ such that 
\begin{enumerate}
	\item $\pi_2(P)\in [\tfrac 34\delta M,\delta M]$, $\pi_2(P')\in[-\tfrac14\delta M,\tfrac14\delta M]$ and $\pi_2(P'')\in[-\delta M,-\tfrac 34\delta M]$, and
	\item the geodesics $\gamma_P$, $\gamma_{P'}$ and $\gamma_{P''}$ are pairwise non-coalescent. 
\end{enumerate}
Since $E_M'\subset E_M$, this completes the proof.
\end{proof}

The following result shows that Lemma~\ref{geodConstrLem} allows us to restrict our attention to potential coalescence of $\gamma_{\ms{m}}$ with geodesics that cross the segment $\{0\}\times[-\delta M,\delta M]$. Recall that if $\gamma$ is an arbitrary oriented path in $G(X)$ and $P\in\gamma$, then $\gamma[P]$ denotes the oriented subpath of $\gamma$ starting at $P$.

\begin{lemma}
\label{geoProdLem}
Let $z\in4\Z\times2\Z$ be such that $z\ne o$ and $\pi_1(z)\le0$. Then, almost surely under the event $E_M$, the following assertion holds. If $x_-\in X\cap(Mz+H^-)$ and $x_+\in X\cap(Mz+H^+)$ are such that 
\begin{enumerate}
\item $[x^-,x^+]$ forms an edge in $G(X)$ that is contained in $B_{M/2}(Mz)$,
\item $\gamma=\gamma^{(M\pi_1(z))}_{x^+}$ exists, and
\item $\gamma[P]\cap \gamma_{\ms{m}}\ne\es$, where $P$ denotes the last intersection point of $\gamma$ and $\ellv_0$,
\end{enumerate}
then $|\pi_2(P)|\le\delta M$.
\end{lemma}

\begin{proof}
In order to derive a contradiction, we assume that $|\pi_2(P)|>\delta M$. Since $\gamma_{\ms{m}}$ is enclosed by the union of $\{0\}\times[-\delta M,\delta M]$, $[x^-_{\ms{u}},x^+_{\ms{u}}]\cup\gamma_{\ms{u}}$ and $[x^-_{\ms{d}},x^+_{\ms{d}}]\cup\gamma_{\ms{d}}$, we deduce that $\gamma[P]$ has a common vertex with $\gamma_{\ms{u}}$ or $\gamma_{\ms{d}}$. We let $x$ be the last such point and assume that it lies on $\gamma_{\ms{u}}$. Then, $\gamma[x]$ and $\gamma_{\ms{u}}[x]$ are two distinct $\hat{u}_0$-geodesics in $G(X)$ with respect to $\nu_1^{(0)}$, contradicting Lemma~\ref{uniGeoLem} and the choice of $\hat{u}_0$.
\end{proof}

\begin{remark}
Note that if $\pi_1(z)=0$, then it is impossible to obtain $|\pi_2(P)|\le\delta M$. Hence, in this case the event $E_M$ already suffices to ensure that $\gamma[P]\cap\gamma_{\ms{m}}=\es$.
\end{remark}

\subsection{The event $A'_{M}$}

The second part of the shield is obtained by changing the Poisson point process in $(3M\times2M)$-box $R_M^-=[-4M,-M]\times[-M,M]$ so as to increase the cost of passing through the rectangle $R_M$. An important feature in our choice of the event $E_M$ is that it only involves geodesics in $G(X)$ with respect to $\nu_1^{(0)}$, but not $\nu_1$. Hence, if the modifications in $R_M^-$ are organized such that they do not influence the configuration of $G(X)$ in $H^+$, then the occurrence of the event $E_M$ is not influenced by this modification. We stress that the latter implication was false if we considered geodesics in $G(X)$ with respect to $\nu_1$ and not $\nu_1^{(0)}$.

First, we introduce a family of events $\{A_{M,1}'\}_{M\ge1}$ guaranteeing that changes of the Poisson point process within $R_M^-$ do not influence the Delaunay triangulation or the relative neighborhood graph outside the dilated rectangle $R_M^-\oplus Q_{8\varepsilon M}(o)$. 
Here we put $\varepsilon=\varepsilon(M)=M^{-31/32}$ and $Q_{8\varepsilon M}(o)=[-4\varepsilon M,4\varepsilon M]^2$. To be more precise, we subdivide $(R_M^-\oplus Q_{M}(o))\setminus R_M^-$ into 
$$K_M=(4\varepsilon^{-1})(3\varepsilon^{-1})-(3\varepsilon^{-1})(2\varepsilon^{-1})=6\varepsilon^{-2}$$
congruent subsquares $Q_i=Q_{\varepsilon M}(v_i)$ of side length $\varepsilon M$, where we assume that $\varepsilon^{-1}$ is an integer. Then, we write $\{X^{(1)}\in A_{M,1}'\}$ if $\#(X^{(1)}\cap Q_i)\in(0,2\varepsilon^2M^2)$ holds for all $1\le i\le K_M$. Since $\#(X^{(1)}\cap Q_i)$ is Poisson-distributed with parameter $\varepsilon^2M^2$, the probability of the event $A_{M,1}'$ tends to $1$ as $M\to\infty$.
Next, we show that under $\{X\in A_{M,1}'\}$ the Delaunay triangulation and the relative neighborhood graph exhibit stabilization outside $R_M^-\oplus Q_{8\varepsilon M}(o)$. 

\begin{lemma}
\label{stabRadLem}
If $X^{(1)}\in A_{M,1}'$, then 
$$G(X^{(1)}\cup\psi)\setminus(R_M^-\oplus Q_{8\varepsilon M}(o))=G(X^{(1)})\setminus(R_M^-\oplus Q_{8\varepsilon M}(o))$$
holds for every finite subset $\psi$ of $R_M^-$.
\end{lemma}
\begin{proof}
We only provide the proof for the Delaunay triangulation, since the case of the relative neighborhood graph is similar but easier. In order to derive a contradiction, assume that we could find finite subsets $\psi,\psi'\subset R_M^-$ and an edge $e$ in $\Del(X^{(1)}\cup\psi)$ such that i) at least one end point of $e$ is outside $R_M^-\oplus Q_{8\varepsilon M}(o)$ and ii) $e$ is not an edge in $\Del(X^{(1)}\cup\psi')$. Then, there exists a disk $D$ containing $e$ and no points of $X^{(1)}\cup\psi$ in its interior. Note that $D$ does not intersect $R_M^-$, since otherwise it would cover one of the cubes $Q_i$, contradicting the assumption that $X^{(1)}\cap Q_i\ne\es$. Hence, $D\cap R_M^-=\es$ and therefore $e$ is also an edge in $\Del(X^{(1)}\cup\psi)$.
\end{proof}

Moreover, we show that under the event $\{A_{M,1}'\}_{M\ge1}$ the sketch provided in Figure~\ref{deleteFig} is accurate.

\begin{lemma}
\label{rmmConfLem}
If $X^{(1)}\in A_{M,1}'$, then 
\begin{enumerate}
\item there exists an edge $[x,y]$ in $\Del(X^{(1)})$ with $\max\{|\pi_1(x)+2M|,|\pi_1(y)+2M|\}\le8\sqrt{\varepsilon}M$, $\pi_2(x)\ge M$ and $\pi_2(y)\le -M$,
\item for every $\rho\in[8\varepsilon,1-8\varepsilon]$ there exists an edge $[x,y]$ in $\Del(X^{(1)})$ such that 
\begin{enumerate}
\item $\max\{|\pi_1(x)-(-1-\rho)M|,|\pi_2(y)-(1-\rho)M|\}\le8\sqrt{\varepsilon}M$, and
\item $0\le\pi_2(x)-M\le8\varepsilon M$ and $0\le\pi_1(y)+M\le8\varepsilon M$.
\end{enumerate}
\end{enumerate}
\end{lemma}

\begin{proof}
Fix $P_0=((-2-8\varepsilon)M,0)$ and consider the disk $D=B_{(1+8\varepsilon)M}(P_0)$ of radius $M+8\varepsilon M$. First, by the choice of the subsquares $Q_i$, there exist $x',y'\in X^{(1)}\cap D$ with $\pi_2(x')\ge M$ and $\pi_2(y')\le-M$. Conversely, we claim that if $x\in X^{(1)}\cap D$ and $|\pi_2(x)|\ge M$, then $|\pi_1(x)+2M|\le8\sqrt{\varepsilon}M$. Indeed,
\begin{align*}
8\varepsilon M\ge|x-P_0|-M=\frac{(\pi_1(x)-(-2-8\varepsilon)M)^2+\pi_2(x)^2-M^2}{\sqrt{(\pi_1(x)-(-2-8\varepsilon)M)^2+\pi_2(x)^2}+M}\ge\frac{(\pi_1(x)-(-2-8\varepsilon)M)^2}{3M},
\end{align*}
so that
$$|\pi_1(x)+2M|\le|\pi_1(x)-(-2-8\varepsilon)M|+8\varepsilon M\le2\sqrt{6\varepsilon} M+8\varepsilon M,$$
which is smaller than $8\sqrt{\varepsilon}M$ if $M$ is sufficiently large. In particular, shrinking the disk $D$ until it contains precisely one $x\in X^{(1)}$ with $\pi_2(x)\ge M$ and one $y\in X^{(1)}$ with $\pi_2(y)\le-M$ proves the first claim.

For the second claim, we proceed similarly, but for the convenience of the reader, we provide some details. For $\rho\in[8\varepsilon,1-8\varepsilon]$, we fix $P_0=((-1-\rho)M ,(1-\rho)M)$ and consider the disk $D=B_{\rho M+8\varepsilon M}(P_0)$ of radius $\rho M+8\varepsilon M$ centered at $P_0$. Again, the choice of the subsquares $Q_i$ implies that there exist $x',y'\in X^{(1)}\cap D$ with $\pi_2(x')\ge M$ and $\pi_1(y')\ge-M$. Conversely, we claim that if $x\in X^{(1)}\cap D$ and $\pi_2(x)\ge M$, then $|\pi_1(x)+(1+\rho)M|\le4\sqrt{\varepsilon}M$. Indeed, as before, 
\begin{align*}
8\varepsilon M\ge \frac{(\pi_1(x)-(-1-\rho)M)^2+(\pi_2(x)-(1-\rho) M)^2-\rho^2M^2}{\sqrt{(\pi_1(x)-(-1-\rho)M)^2+(\pi_2(x)-(1-\rho)M)^2}+\rho M}\ge\frac{(\pi_1(x)-(-1-\rho)M)^2}{2M},
\end{align*}
so that
$$|\pi_1(x)-(-1-\rho)M|\le4\sqrt{\varepsilon}M.$$
By an analogous computation, we see that if $y\in X^{(1)}\cap D$ and $\pi_1(y)\ge-M$, then $|\pi_2(y)-(1-\rho)M|\le4\sqrt{\varepsilon}M$. Similar as in the previous case, shrinking $D$ until it contains precisely one $x\in X^{(1)}$ with $\pi_2(x)\ge M$ and one $y\in X^{(1)}$ with $\pi_1(y)\ge -M$ constructs the desired edge.
\end{proof}

\begin{remark}
\label{diagStrLem}
In particular, under the event $A_{M,1}'$, if $\gamma$ is a path in $G(X)$ that starts in $H^-\setminus R_M$, intersects the segment $\{0\}\times[-\delta M,\delta M]$ but does not intersect the area $[-(1+8\sqrt{\varepsilon})M,0]\times(\R\setminus[-M,M])$, then $\gamma$ intersects the union of segments $[(-2-8\sqrt{\varepsilon})M,(-1+8\sqrt{\varepsilon})M]\times\{\pm M\}$.
\end{remark}

Moreover, we obtain very precise control over the behavior of edges in $\Del(X^{(1)})$ intersecting the horizontal segment $[-2M,-M]\times\{M\}$.

\begin{corollary}
\label{diagDelCor}
Assume $X^{(1)}\in A_{M,1}'$ and that $e$ is an edge of $\Del(X^{(1)})$ intersecting the segment $[(-2+\delta)M,(-1-\delta)M]\times\{M\}$ at some point $P_1$ and write $\pi_1(P_1)=(-1-\rho)M$. Then, $e$ intersects the vertical segment 
$$\{-M\}\times[(1-\rho-32\sqrt{\varepsilon})M,(1-\rho+32\sqrt{\varepsilon})M],$$
and therefore the length of $e$ is bounded below by $(\sqrt{2}\rho -32\sqrt{\varepsilon})M$. 
\end{corollary}

\begin{proof}
 By Lemma~\ref{rmmConfLem} there exist $x,x',y,y'\in X$ such that $[x,y]$ and $[x',y']$ are edges in $\ms{Del}(X)$ with the following properties.
\begin{enumerate}
\item $\max\{|\pi_1(x)-(-1-\rho -16\sqrt{\varepsilon})M|,|\pi_2(y)-(1-\rho -16\sqrt{\varepsilon})M|\}\le8\sqrt{\varepsilon}M$,
\item $\max\{|\pi_1(x')-(-1-\rho +16\sqrt{\varepsilon})M|,|\pi_2(y')-(1-\rho +16\sqrt{\varepsilon})M|\}\le8\sqrt{\varepsilon}M$,
\item $\pi_2(x),\pi_2(x')\in[M,(1+8\varepsilon) M]$ and $\pi_1(y),\pi_1(y')\in[-M,(-1+8\varepsilon )M]$.
\end{enumerate}
If we let $P_2$ denote the intersection point of $[x,y]$ with the horizontal segment $[-2M,-M]\times\{M\}$, then 
$$|\pi_1(P_2)-(-1-\rho)M|\in[-26\sqrt{\varepsilon}M,-6\sqrt{\varepsilon}M].$$
Similarly, if $P_2'$ denotes the intersection point of $[x',y']$ with the horizontal segment $[-2M,-M]$, then
$$|\pi_1(P_2')-(-1-\rho)M|\in[6\sqrt{\varepsilon}M,26\sqrt{\varepsilon}M].$$
As similar relations can be obtained for the intersections $P_3$, $P_3'$ of $[x,y]$ and $[x',y']$ with the vertical segment $\{-M\}\times[0,M]$, we see that $e$ is trapped between $[x,y]$ and $[x',y']$. In particular, $e$ intersects the vertical segment $\{-M\}\times(\pi_2(P_3),\pi_2(P_3'))$. Since,
$$(1-\rho -26\sqrt{\varepsilon})M\le \pi_2(P_3)\le\pi_2(P_3')\le (1-\rho +26\sqrt{\varepsilon})M$$
this completes the proof.
\end{proof}

\begin{remark}
\label{diagDelRem}
A similar result holds if $e$ intersects the segment $[(-2-\delta)M,(-2+\delta)M]\times\{M\}$. However, for those edges there are two options. Either they intersect the vertical segment $\{-M\}\times\R$ close to the point $(-M,0)$ or they intersect the horizontal segment $\R\times\{-M\}$ close to the point $(-2M,-M)$.
\end{remark}

\begin{remark}
\label{diagRngRem}
Since the relative neighborhood graph is a subgraph of the Delaunay triangulation, Corollary~\ref{diagDelCor} and the previous remark can be used to show that if $X^{(1)}\in A_{M,1}'$, then $\Rng(X^{(1)})$ does not contain an edge intersecting the segment $[(-2-\delta)M,(-1-\delta)M]\times\{M\}$.
\end{remark}

In addition to the stabilization property, it will also be important to know that shortest path-lengths of $G(X^{(1)})$ in the annulus $(R_M^-\oplus Q_M(o))\setminus R_M^-$ are not too long. Since shortest-path lengths in the relative neighborhood graph are closely related to descending chains~\cite[Lemmas 10]{aldLin}, we first recall the notion of descending chains from~\cite{lilypond}. For $b\ge1$ we say that a sequence of distinct vertices $X_{i_1},\ldots,X_{i_n}\in X^{(1)}$ forms a \emph{$b$-bounded descending chain} if 
$$b\ge|X_{i_1}-X_{i_2}|\ge\cdots\ge|X_{i_{n-1}}-X_{i_n}|.$$
Now, we say that the event $A_{M,2}'$ occurs if every $16\varepsilon M$-bounded descending chain of $X^{(1)}$ starting in $R_M^-\oplus Q_M(o)$ consists of at most $10^4\varepsilon^2 M^2$ hops. Finally, we put $A_M'=A_{M,1}'\cap A_{M,2}'$. Using the results from~\cite{lilypond}, we show that the events $A_M'$ occur with high probability.

\begin{lemma}
\label{highProbLem1}
It holds that $\lim_{M\to\infty}\P(X^{(1)}\in A_M')=1$. 
\end{lemma}

\begin{proof}
Since we have already seen that $\P(A_{M,1}')$ tends to $1$ as $M\to\infty$, it suffices to show that $\lim_{M\to\infty}\P(A_{M,2}')=1$. First, monotonicity of descending chains allows us to prove this result when $X$ is replaced by $X^{(1)}$. But then the computations in~\cite[Section 3.2]{lilypond} show that the probability of obtaining a $16\varepsilon M$-bounded descending chain of $X^{(1)}$ starting in $R_M^-\oplus Q_M(o)$ and consisting of more than $10^4\varepsilon^2 M^2$ hops is at most 
$$\nu_2(R_M^-\oplus Q_M(o))\frac{(16^2\pi\varepsilon^2 M^2)^{10^4\varepsilon^2M^2}}{(10^4\varepsilon^2M^2)!}.$$
By Stirling's formula, this expression tends to $0$ as $M\to\infty$.
\end{proof}

We show now that under the event $A_{M,2}'$ shortest-path lengths in the annulus $(R_M^-\oplus Q_M(o))\setminus R_M^-$ are not too long.

\begin{lemma}
\label{rmmConfLem2}
If $X^{(1)}\in A_{M,2}'$ and $x,x'\in X^{(1)}\cap(R_M^-\oplus Q_{M/2}(o))$ are such that $|x-x'|\le 16\varepsilon M$, then $x$ and $x'$ can be connected by a path in $\Rng(X^{(1)})$ of length at most $\sqrt{M}$.
\end{lemma}

\begin{proof}
	Since $X^{(1)}\in A_M$, there does not exist a $16\varepsilon M$-bounded descending chain starting from $x$ and consisting of more than $10^4\varepsilon^2M^2$ hops. Hence, proceeding as in~\cite[Lemmas 10]{aldLin}, we see that $x$ and $x'$ can be connected by a path in $\Rng(X^{(1)})$ that is contained in $B_{10^6\varepsilon^3M^3}(x)$. Since $X^{(1)}\in A_M'$, the total number of Poisson points in $B_{10^6\varepsilon^3M^3}(x)$ is at most $10^{13}\varepsilon^6M^6$. Using that the maximum degree of the relative neighborhood graph is at most 6, it follows that $x$ and $x'$ can be connected by a path in $\Rng(X^{(1)})$ of length at most $10^{20}\varepsilon^9M^9$. By the choice of $\varepsilon=\varepsilon(M)$, this quantity is less than $\sqrt{M}$, provided that $M$ is sufficiently large.
\end{proof}

\subsection{The event $A''_{M}$}

The event $A''_{M}$ has to encode a certain control over shortest-path lengths on $G(X)$.

For any $P,P'\in G(X)$ recall that we let $\ell(P,P')=\ell_{G(X)}(P,P')$ denote the Euclidean length of the shortest path on $G(X)$ connecting $P$ and $P'$. It is shown in~\cite{zuyevDel} for the Delaunay triangulation and in~\cite{aldLin} for the relative neighborhood graph that there exists a deterministic value $\mu\in[1,4/\pi]$, called \emph{time constant}, such that almost surely
$$
\mu = \lim_{n\to\infty}n^{-1}{\E\ell(o,ne_1)} ~.
$$
If $o$ or $ne_1$ are not contained on $G(X)$, then we put $\ell(o,ne_1)=\ell(q(o),q(ne_1))$, where $q(o)$ and $q(ne_1)$ denote the closest points of $X$ to $o$ and $ne_1$, respectively.
In the following, we put $\mu_-=(1-\delta')\mu$ and $\mu_+=(1+\delta')\mu$, where $\delta'\in(0,1)$ is a small number that will be fixed in the proof of Lemma~\ref{fImplLem} below. First, we construct a family of events $\{A_M''\}_{M\ge1}$ such that $\lim_{M\to\infty}\P(X^{(1)}\in A_M'')=1$ and, almost surely, if $X^{(1)}\in A_M''$, then the following properties are satisfied, where we consider the subsquares $Q_{\varepsilon M}(v_i)$, $i\in\{1,\ldots,K_M\}$ introduced in the paragraph preceding Lemma~\ref{stabRadLem}:
\begin{enumerate} 
\item[{\bf (D1)}] if $Q_{\varepsilon M}(v_i)\cap Q_{\varepsilon M}(v_j)\ne\es$, $Q_{\varepsilon M}(v_i)\cap(R_M^-\oplus Q_{16\varepsilon M}(o))=\es$ and $Q_{\varepsilon M}(v_j)\cap(R_M^-\oplus Q_{16\varepsilon M}(o))=\es$ then $\ell_{G(X^{(1)})}(v_i,v_j)\le \mu_+|v_i-v_j|$
\item[{\bf (D2)}] if $r\ge M$ and $P,P'\in G(X^{(1)})\cap B_r(o)$ are such that $|P-P'|\ge{\delta r}$, then the following properties are satisfied:
\begin{enumerate}
	\item[(a)] if $([P,P']\oplus B_{r^{7/8}}(o))\cap R_M^-=\es$, then $\ell_{G(X^{(1)})}(P,P')\le \mu_+|P-P'|$,
	\item[(b)] if the geodesic from $P$ to $P'$ in $G(X^{(1)})$ does not hit the set $R_M^-\oplus Q_{8\varepsilon M}(o)$, then $\ell_{G(X^{(1)})}(P,P')\ge \mu_-|P-P'|$,
\end{enumerate}
\end{enumerate}

\begin{lemma}
\label{highProbLem2}
There exists a family of events $\{A_M''\}_{M\ge1}$ such that $\lim_{M\to\infty}\P(X^{(1)}\in A''_M)=1$ and for which conditions {\bf (D1)} and {\bf (D2)} are satisfied. 
\end{lemma}

\begin{proof}
Suppose that we could construct a family of events $\{A_M^*\}_{M\ge1}$ such that $\lim_{M\to\infty}\P(X\in A_M^*)=1$ and such that if $X\in A_M^*$, then the following condition is satisfied:
\begin{enumerate}
\item[{\bf (D1')}] if $Q_{\varepsilon M}(v_i)\cap Q_{\varepsilon M}(v_j)\ne\es$, then $q(v_i)$ and $q(v_j)$ are connected by a path $\gamma$ in $G(X)$ such that $\nu_1(\gamma)\le\mu_+|v_i-v_j|$ and $\gamma$ is contained in $(Q_{\varepsilon M}(v_i)\cup Q_{\varepsilon M}(v_j))\oplus B_{3\varepsilon M}(o)$. 
\item[{\bf (D2')}] for all integers $n\ge M$ and all $P,P'\in G(X)\cap B_{2n}(o)$ with $|P-P'|\ge\tfrac12 \delta n$ it holds that 
\begin{enumerate}
\item[(a)] $P$ and $P'$ can be connected by a path $\gamma$ in $G(X)$ such that $\nu_1(\gamma)\le\mu_+|P-P'|$ and $\gamma$ is contained in $[P,P']\oplus B_{n^{7/8}/2}(o)$. 
\item[(b)] $\ell_{G(X)}(P,P')\ge\mu_-|P-P'|$,
\end{enumerate}
\end{enumerate}

As in~\cite[Lemma 8]{efpp}, Fubini's theorem produces a configuration $\varphi\subset R_M^-$ such that $\P(\varphi\cup X^{(1)}\in A_M^*)\ge\P(X\in A_M^*)$. Then, we let $\{X^{(1)}\in A_M''\}$ denote the event that $\varphi\cup X^{(1)}\in A_M^*$ and $X^{(1)}\in A_M'$. In particular, $\lim_{M\to\infty}\P(X^{(1)}\in A_M'')=1$. Moreover, we claim that if $X^{(1)}\in A_M''$, then conditions {\bf (D1)} and {\bf (D2)} are satisfied. 
Regarding condition {\bf (D1)} let $v_i$ and $v_j$ be subsquare centers satisfying the desired conditions. Then, by {\bf (D1')}, $q(v_i)$ and $q(v_j)$ are connected by a path $\gamma$ in $G(X^{(1)}\cup \varphi)$ such that $\nu_1(\gamma)\le\mu_+|v_i-v_j|$ and $\gamma$ is contained in $(Q_{\varepsilon M}(v_i)\cup Q_{\varepsilon M}(v_j))\oplus B_{3\varepsilon M}(o)$. Since we know that $X^{(1)}\in A_{M,1}'$, we conclude from Lemma~\ref{stabRadLem} that $\gamma$ is also a path in $G(X^{(1)})$, so that $\ell_{G(X^{(1)})}(v_i,v_j)\le \mu_+|v_i-v_j|$.
For condition {\bf (D2)} we may argue similarly. Indeed, let $r\ge M$ and $P,P'\in G(X^{(1)})\cap B_r(o)$ be as in condition {\bf (D2)}. If $([P,P']\oplus B_{r^{7/8}}(o))\cap R_M^-=\es$, then, by condition {\bf (D2')} for $n=\lceil r\rceil$, the points $P,P'$ can be connected by a path $\gamma$ in $G(X^{(1)}\cup\varphi)$ such that $\nu_1(\gamma)\le\mu_+|P-P'|$ and $\gamma$ is contained in $[P,P']\oplus B_{n^{7/8}/2}(o)$. Moreover, by Lemma~\ref{stabRadLem}, $\gamma$ is also a path in $G(X^{(1)})$.
Next, we assume that the geodesic from $P$ to $P'$ in $G(X^{(1)})$ does not hit the set $R_M^-\oplus Q_{8\varepsilon M}(o)$. In particular, again using Lemma~\ref{stabRadLem}, we deduce that this geodesic is also a path in $G(X^{(1)}\cup \varphi)$. Therefore, condition {\bf (D2')} implies that
$$\ell_{G(X^{(1)})}(P,P')\ge \ell_{G(X^{(1)}\cup\varphi)}(P,P')\ge \mu_-|P-P'|.$$
Hence, it remains to construct the family $\{A_M^*\}_{M\ge1}$ with the desired properties. Let $z,z'\in\Z^d\cap B_{3n}(o)$ be arbitrary. Then, it is shown in~\cite[Theorems 1 and 2]{modDev} that the probability that there exist $P\in G(X)\cap Q_1(z)$ and $P'\in G(X)\cap Q_1(z')$ such that
\begin{enumerate}
\item $P$ and $P'$ cannot be connected by a path $\gamma$ in $G(X)$ such that $\nu_1(\gamma)\le\mu_+|P-P'|$ and $\gamma$ is contained in $[P,P']\oplus B_{n^{7/8}/2}(o)$, or 
\item $\ell_{G(X)}(P,P')<\mu_-|P-P'|$
\end{enumerate}
decays at  stretched exponential speed in $|z-z'|$. Since the total number of $z,z'\in \Z^2$ such that $Q_1(z)\cap B_{2n}(o)\ne\es$ and $Q_1(z')\cap B_{2n}(o)\ne\es$ grows polynomially in $n$, this completes the proof.
\end{proof}

\subsection{How $A_M'\cap A_M''\cap E_M$ implies $F_M$}

Now, we put $A_M= A_M'\cap A_M''\cap E_M$ and provide a key deterministic argument which shows that $X^{(1)}\in A_M$ implies $X^{(1)}\in F_M$.

\begin{lemma}
\label{fImplLem}
Let $A_M= A_M'\cap A_M''\cap E_M$. Then $\{X^{(1)}\in F_M\}$ holds almost surely under the event $\{X^{(1)}\in A_M\}$.
\end{lemma}

As noted in the beginning of this section, once Lemma~\ref{fImplLem} is established, the proof of Proposition~\ref{part2Lem} is complete. 

To prove Lemma~\ref{fImplLem}, we proceed by contradiction. Hence, using Lemma~\ref{geoProdLem}, we assume that there exists $z\in4\Z\times2\Z$, $x_-\in X^{(1)}\cap(Mz+H^-)$ and $x_+\in X^{(1)}\cap(Mz+H^+)$ such that 
\begin{enumerate}
\item $n=\pi_1(z)<0$ and $[x^-,x^+]$ forms an edge in $G(X^{(1)})$ that is contained in $B_{M/2}(Mz)$,
\item $\gamma=\gamma^{(nM)}_{x^+}$ exists, $\gamma[P]$ coalesces with $\gamma_{\ms{m}}$ and $|\pi_2(P)|\le\delta M$, where $P$ denotes the last intersection point of $\gamma$ and $\ellv_0$.
\end{enumerate}
Note that the reduction to the case $n<0$ is a consequence of the remark after Lemma~\ref{geoProdLem}.

Now, we distinguish several cases. First, assume that there exists $\xi\in[0,1-\delta]$ such that $\gamma$ intersects the union of rays $\{-\xi M\}\times(\R\setminus[-M,M])$ at some point $P_1$. Without loss of generality, we assume that $\pi_2(P_1)\ge M$ and that $P_1$ is the last such intersection point. Furthermore, let $Q$ denote a point of $\gamma[P]$ satisfying $\pi_1(Q)=m\pi_2(P_1)-\xi M$. Here, $m\ge2$ is a sufficiently large integer that will be fixed in the course of the proof. Figure~\ref{case1Fig} provides a rough illustration. In particular, it does not show edges crossing $R^{-}_{M}$.

\begin{figure}[!htpb]
\centering
\begin{tikzpicture}[scale=1]
\draw (-4,-1)--(0,-1)--(0,1)--(-4,1)--(-4,-1);
\draw[dashed] (-1,1)--(-1,0);
\draw[dashed] (-0.6,1)--(-0.6,3);
\draw[dashed] (-0.6,1)--(-0.6,0);
\draw[dashed] (-1,1)--(-1,-1);
\draw[dashed]  plot [smooth, tension=1] coordinates { (-0.6,2) (0.8,1.5) (1.5,1) (3.4,0.1) };
\fill (-0.6,2) circle (2pt);
\fill (3.4,0.1) circle (2pt);
\fill (0,0) circle (2pt);
\coordinate[label=-0:$P_1$] (B) at (-0.6,2);
\coordinate[label=-0:$P_2$] (F) at (-0.6,0.4);
\fill (F) circle (2pt);
\coordinate[label=0:$Q$] (E) at (3.4,0.1);
\coordinate (C) at (0,0.15);
\coordinate (CC) at (0,-0.15);
\coordinate[label=-135:$P$] (CCC) at (0,0);
\draw[decorate,decoration={brace,mirror}] (-0,-1.1) -- (3.4,-1.1);
\draw[decorate,decoration={brace,mirror}] (-0.6,-1.1)--(-0,-1.1) ;
\coordinate[label=180:\textcolor{blue}{$\gamma_{\ms{u}}$}] (Q) at (2.4,0.4);
\coordinate[label=180:\textcolor{blue}{$\gamma_{\ms{d}}$}] (Q) at (2.4,-0.3);
\coordinate[label=180:$\gamma$] (Q) at (-0.6,1.3);
\coordinate[label=-90:$m\pi_2(P_1)-\xi M$] (Q) at (1.7,-1.1);
\coordinate[label=-90:$\xi M$] (Q) at (-0.2,-1.1);
\draw  plot [smooth, tension=0.8] coordinates { (-0.5,2.5) (-0.6,2) (-0.8,1) (F)   (CCC)  (E) };
\draw[blue,decoration = {coil, aspect=0},decorate,very thin] (C)--(4,0.3);
\draw[blue,decoration = {coil, aspect=0},decorate,very thin] (CC)--(4,-0.1);
\end{tikzpicture}
\caption{Illustration of the event $A_M^{(1)}$}
\label{case1Fig}
\end{figure}

Let $A_M^{(1)}$ denote the event that $A_M\cap F_M^c$ occurs and that a point $P_1$ as described above exists. Under the event $A_M^{(1)}$ we derive upper and lower bounds on the shortest-path length $\ell^{(-M)}(P_1,Q)$ that cannot be satisfied simultaneously. 

\begin{lemma}
\label{upBound1Lem}
Almost surely under the event $A_M^{(1)}$,
$$\ell^{(-M)}(P_1,Q)\le \mu_+((m+m^{-1})\pi_2(P_1)+16\delta m \pi_2(P_1)).
$$
\end{lemma}

\begin{lemma}
\label{lowBound10Lem}
Almost surely under the event $A_M^{(1)}$,
$$\ell^{(nM)}(P,Q)\ge\mu_-(m\pi_2(P_1)-\xi M-17\delta m\pi_2(P_1)).$$
\end{lemma}

In order to derive lower bounds for $\ell^{(nM)}(P_1,P)$ we need to further decompose the event $A^{(1)}_M$. More precisely, let $A_M^{(1,a)}$ denote the event that $A_M^{(1)}$ occurs and that $\gamma[P_1]$ stays within the vertical half-plane $[-\xi M,\infty)\times\R$. Then, we have the following lower bounds for $\ell^{(nM)}(P_1,P)$. 
\begin{lemma}
\label{lowBound1aLem}
Almost surely under the event $A_M^{(1,a)}$,
$$\ell^{(nM)}(P_1,P)\ge\mu_-((\sqrt{2}-1)\pi_2(P_1)+(\xi-\delta)M).$$
\end{lemma}
Conversely, we put $A_M^{(1,b)}=A_M^{(1)}\setminus A^{(1,a)}_M$ and observe that under $A^{(1,b)}_M$ the path $\gamma[P_1,P]$ intersects the segment $\{-\xi M\}\times[-M,M]$. We let $P_2=(-\xi M,\eta M)$ denote the last such intersection point.
\begin{lemma}
\label{lowBound1bLem}
Almost surely under the event $A_M^{(1,b)}$,
$$\ell^{(nM)}(P_1,P)\ge\mu_-(\pi_2(P_1)+(\xi-2)M+\sqrt{1+\eta^2}M-2\delta M)+(1-\eta)M.$$
\end{lemma}

Before we provide proofs of Lemmas~\ref{upBound1Lem}--\ref{lowBound1bLem}, we show to deduce from them a contradiction.

\begin{lemma}
\label{contr1Lem}
There exists $m_0\ge1$ such that for all $m\ge m_0$, all sufficiently small $\delta,\delta'\in(0,1)$ and all sufficiently large $M\ge1$ it holds that $\P\big(A_M^{(1)}\big)=0$.
\end{lemma}
\begin{proof}
We start by showing $\P\big(A_M^{(1,a)}\big)=0$. First, we recall that $\gamma$ is a geodesic in $G(X^{(1)})$ with respect to $\nu^{(nM)}_1$, so that 
$$\ell^{(-M)}(P_1,Q)\ge \ell^{(nM)}(P_1,Q)=\ell^{(nM)}(P_1,P)+\ell^{(nM)}(P,Q).$$
Hence, Lemmas~\ref{upBound1Lem},~\ref{lowBound10Lem} and~\ref{lowBound1aLem} give that 
$$\mu_-\pi_2(P_1)(m+\sqrt{2}-1-18\delta m)\le \mu_+((m+m^{-1})\pi_2(P_1)+16\delta m\pi_2(P_1)),$$
so that
$$\mu_-(\sqrt{2}-1)\le m(34\mu_+\delta+\mu_+-\mu_-)+m^{-1}\mu_+.$$
Since $\sqrt{2}>1$ this yields a contradiction if first $m\ge1$ is chosen sufficiently large and then $\delta,\delta'\in(0,1)$ are chosen sufficiently small.

To show $\P(A_M^{(1,b)})=0$, we proceed similarly. Using Lemmas~\ref{upBound1Lem},~\ref{lowBound10Lem} and~\ref{lowBound1bLem}, we obtain that 
\begin{align*}&\mu_-((m+1)\pi_2(P_1)-2M+\sqrt{1+\eta^2}M)+(1-\eta)M\\
\quad&\le \mu_+((m+m^{-1})\pi_2(P_1)+35\delta m\pi_2(P_1)),
\end{align*}
which gives that 
\begin{align*}
\big(35\mu_+\delta m+\mu_+(m+m^{-1})-\mu_-m\big)\ge\mu_-+\frac{M\mu_-}{\pi_2(P_1)}\Big(\frac{1-\eta}{\mu_-}+\sqrt{1+\eta^2}-2\Big).
\end{align*}
First, the left-hand side becomes arbitrarily small if first $m\ge1$ is chosen sufficiently large and then $\delta,\delta'\in(0,1)$ are chosen sufficiently small. Moreover, the right-hand side is bounded below by
$$\mu_-\Big(1-\frac{M}{\pi_2(P_1)}\Big(1-\frac{1-\eta}{\mu_-}-\sqrt{1+\eta^2}+1\Big)\Big)\ge\mu_-\min\Big\{1,\frac{1-\eta}{\mu_-}+\sqrt{1+\eta^2}-1\Big\}.$$
Now a quick computation shows that the right-hand side is bounded below by $1/4$, which gives the desired contradiction. 
\end{proof}

Now, we prove Lemmas~\ref{upBound1Lem}--\ref{lowBound1bLem}.
\begin{proof}[Proof of Lemma~\ref{upBound1Lem}]
Condition {\bf (D2)} gives the upper bound $\ell^{(-M)}(P_1,Q)\le\mu_+|P_1-Q|$, so that we obtain that
\begin{align*}
\ell^{(-M)}(P_1,Q)&\le\mu_+(\sqrt{m^2+1}\pi_2(P_1)+|\pi_2(Q)|)\\
&\le \mu_+\big(m\pi_2(P_1)+\frac{1}{m+\sqrt{m^2+1}}\pi_2(P_1)+|\pi_2(Q)|\big)\\
&\le  \mu_+((m+m^{-1})\pi_2(P_1)+|\pi_2(Q)|).
\end{align*}
Now, by uniqueness of $\hat{u}_0$-geodesics in $G^{}(X^{(1)})$ with respect to $\nu^{(0)}_1$, the geodesic $\gamma[P]$ is trapped between $\gamma_{\ms{u}}$ and $\gamma_{\ms{d}}$. Therefore, an argument based on elementary geometry shows that
$$|\pi_2(Q)|\le\tan(4\delta)m \pi_2(P_1)+\frac{2\delta M}{\cos(4\delta)}\le16\delta m \pi_2(P_1),$$
so that
$$\ell^{(-M)}(P_1,Q)\le\mu_+(m+m^{-1})\pi_2(P_1)+16\delta m\mu_+\pi_2(P_1),$$
as required.
\end{proof}

\begin{proof}[Proof of Lemma~\ref{lowBound10Lem}]
Let $Q'$ denote the point with coordinates $(m\pi_2(P_1)-\xi M,0)$. Then, condition {\bf (D2)} gives that
$$\ell^{(nM)}(P,Q)\ge\mu_-|P-Q|\ge\mu_-(|Q'|-|P|-\pi_2(Q))\ge\mu_-(m\pi_2(P_1)-\xi M-17\delta m\pi_2(P_1)),$$
as required.
\end{proof}

\begin{proof}[Proof of Lemma~\ref{lowBound1aLem}]
Since $\gamma[P_1]$ is contained in $[-\xi M,\infty)\times\R$,  condition {\bf (D2)} shows that 
	$$\ell^{(nM)}(P_1,P)\ge\mu_-(|P_1|\hspace{0.05cm}-\hspace{0.05cm}|P|)\ge\mu_-(\sqrt{\pi_2(P_1)^2+\xi^2M^2}-\delta M)\ge\mu_-((\sqrt{2}-1)\pi_2(P_1)+(\xi-\delta)M),$$
where the last inequality follows from the observation that the function $x\mapsto\sqrt{\pi_2(P_1)^2+x^2}-x$ is decreasing.
\end{proof}

\begin{proof}[Proof of Lemma~\ref{lowBound1bLem}]
	Let $P_2'$ denote the first point on $\gamma[P_1,P]$ with $\pi_2(P_2')=(1+\delta)M$. If $\pi_2(P_1)\le (1+\delta)$, we put $P_2'=P_1$. Now, condition {\bf (D2)} allows us to conclude that
$$\ell^{(nM)}(P_1,P_2)\ge\mu_-|P_1-P_2'|+|P_2'-P_2|\ge \mu_-(\pi_2(P_1)-M-\delta M)+(1-\eta)M.$$
Similarly,
$$\ell^{(nM)}(P_2,P)\ge\mu_-|P_2-P|\ge\mu_-(|P_2|-\delta M).$$
In particular, combining the lower bounds above yields
\begin{align*}
\ell^{(nM)}(P_1,P)&\ge\mu_-(\pi_2(P_1)-M+\sqrt{\xi^2+\eta^2}M-2\delta  M)+(1-\eta)M\\
&\ge\mu_-(\pi_2(P_1)+(\xi-2)M+\sqrt{1+\eta^2}M-2\delta  M)+(1-\eta)M,
\end{align*}
as required.
\end{proof}
Hence, in the following we may assume that $\gamma$ does not hit the area $[(-1+\delta)M,0]\times(\R\setminus [-M,M])$. In particular, it crosses the union of segments $[(-2-\delta) M,(-1+\delta)M]\times\{\pm M\}$. We let $P_1=(-\xi M,\pi_2(P_1))$ denote the last such intersection point and assume that $\pi_2(P_1)=M$. Next, we let $Q$ denote a point on $\gamma[P]$ satisfying $\pi_1(Q)=(m-1)M$, where $m\ge3$ will again be a sufficiently large integer. We refer the reader to Figure~\ref{case2Fig} for an illustration.

\begin{figure}[!htpb]
\centering
\begin{tikzpicture}[scale=1]
\draw (-4,-1)--(0,-1)--(0,1)--(-4,1)--(-4,-1);
\draw[dashed] (-1,1)--(-1,-1);
\draw[red] (-3,1)--(-3,-1);
\draw[red] (-2.1,1)--(-2.1,-1);
\draw[red] (-2.2,1)--(-2.2,-1);
\draw[red] (-2.3,1)--(-2.3,-1);
\draw[red] (-2.4,1)--(-2.4,-1);
\draw[red] (-2.5,1)--(-2.5,-1);
\draw[red] (-2.6,1)--(-2.6,-1);
\draw[red] (-2.7,1)--(-2.7,-1);
\draw[red] (-2.8,1)--(-2.8,-1);
\draw[red] (-2.9,1)--(-2.9,-1);
\draw[red] (-2,1)--(-2,-1);
\draw[red] (-3,-1)--(-4,0);
\draw[red] (-3.1,-1)--(-4,-0.10);
\draw[red] (-3.2,-1)--(-4,-0.20);
\draw[red] (-3.3,-1)--(-4,-0.30);
\draw[red] (-3.4,-1)--(-4,-0.40);
\draw[red] (-3.5,-1)--(-4,-0.50);
\draw[red] (-3.6,-1)--(-4,-0.60);
\draw[red] (-3.7,-1)--(-4,-0.70);
\draw[red] (-3.8,-1)--(-4,-0.80);
\draw[red] (-3.9,-1)--(-4,-0.90);
\draw[red] (-3.1,1)--(-4,0.10);
\draw[red] (-3.2,1)--(-4,0.20);
\draw[red] (-3.3,1)--(-4,0.30);
\draw[red] (-3.4,1)--(-4,0.40);
\draw[red] (-3.5,1)--(-4,0.50);
\draw[red] (-3.6,1)--(-4,0.60);
\draw[red] (-3.7,1)--(-4,0.70);
\draw[red] (-3.8,1)--(-4,0.80);
\draw[red] (-3.9,1)--(-4,0.90);
\draw[red] (-1.1,1)--(-1,0.90);
\draw[red] (-1.2,1)--(-1,0.80);
\draw[red] (-1.3,1)--(-1,0.70);
\draw[red] (-1.4,1)--(-1,0.60);
\draw[red] (-1.5,1)--(-1,0.50);
\draw[red] (-1.6,1)--(-1,0.40);
\draw[red] (-1.7,1)--(-1,0.30);
\draw[red] (-1.8,1)--(-1,0.20);
\draw[red] (-1.9,1)--(-1,0.10);
\draw[red] (-1.1,-1)--(-1,-0.90);
\draw[red] (-1.2,-1)--(-1,-0.80);
\draw[red] (-1.3,-1)--(-1,-0.70);
\draw[red] (-1.4,-1)--(-1,-0.60);
\draw[red] (-1.5,-1)--(-1,-0.50);
\draw[red] (-1.6,-1)--(-1,-0.40);
\draw[red] (-1.7,-1)--(-1,-0.30);
\draw[red] (-1.8,-1)--(-1,-0.20);
\draw[red] (-1.9,-1)--(-1,-0.10);

\coordinate[label=180:\textcolor{blue}{$\gamma_{\ms{u}}$}] (Q) at (1.8,0.4);
\coordinate[label=180:\textcolor{blue}{$\gamma_{\ms{d}}$}] (Q) at (1.8,-0.3);
\coordinate[label=-135:$P$] (CCC) at (0,0);
\coordinate[label=0:$Q$] (E) at (3.4,0.1);
\fill (0,0) circle (2pt);
\fill (3.4,0.1) circle (2pt);
\fill (-1.5,1.0) circle (2pt);
\coordinate[label=180:$P_1$] (Q) at (-1.5,1.2);
\coordinate (C) at (0,0.15);
\coordinate (CC) at (0,-0.15);
\draw  plot [smooth, tension=0.4] coordinates { (-1.5,1.2) (-1.5,1.0) (-1.0,0.5) (0,0) (3.4,0.1) };
\draw[dashed]  plot [smooth, tension=0.4] coordinates {  (-1.5,1.0) (-1.0,1.0) (3.4,0.1) };
\draw[blue,decoration = {coil, aspect=0},decorate,very thin] (C)--(4,0.3);
\draw[blue,decoration = {coil, aspect=0},decorate,very thin] (CC)--(4,-0.1);
\end{tikzpicture}
\caption{Illustration of the second case}
\label{case2Fig}
\end{figure}

The proof of Lemma~\ref{fImplLem} in the current setting is similar to what we have seen above, but for the convenience of the reader, we include some details. More precisely, putting $A_M^{(2)}=A_M\cap F_M^c\setminus A_M^{(1)}$  we derive contradictory upper and lower bounds on the shortest-path length $\ell^{(-M)}(P_1,Q)$.

\begin{lemma}
\label{upBound2Lem}
Almost surely under the event $A_M^{(2)}$ it holds that
$$\ell^{(-M)}(P_1,Q)\le \mu_+(m+m^{-1}+\xi -1+20\delta m)M.$$
\end{lemma}

To derive the lower bounds we introduce a further auxiliary point $P_2=(\pi_1(P_2),\eta M)$ as last point of $\gamma[P_1]$ satisfying $\pi_1(P_2)=-(1-\delta)M$. First, we provide lower bounds for $\ell^{(nM)}(P_1,P_2)$ and $\ell^{(nM)}(P_2,Q)$.
\begin{lemma}
\label{lowBound20Lem}
Almost surely under the event $A_M^{(2)}$ it holds that
$$\ell^{(nM)}(P_1,P_2)\ge(1-\eta)M,$$
and
$$\ell^{(nM)}(P_2,Q)\ge\mu_-(m-1+\sqrt{\eta^2+1}-20\delta m)M.$$
\end{lemma}
If we let $A_M^{(2,a)}$ denote the event that $A_M^{(2)}$ occurs and $\xi\le1+4\delta$, then the lower bounds derived in Lemma~\ref{lowBound20Lem} are sufficient to arrive at the desired contradiction. However, 
if the event $A_M^{(2,b)}=A_M^{(2)}\setminus A_M^{(2,a)}$ occurs, then a more refined reasoning is necessary. It is important to observe that by Remark~\ref{diagRngRem} this case cannot occur if $G(X^{(1)})=\Rng(X^{(1)})$, so that we may assume $G(X^{(1)})=\Del(X^{(1)})$. 
\begin{lemma}
\label{lowBound2bLem}
Almost surely under the event $A_M^{(2,b)}$ it holds that
$$\ell^{(nM)}(P_1,P_2)\ge(\sqrt{2}(\xi-1)+|2-\xi-\eta|-2\delta)M.$$
\end{lemma}

Before we provide proofs of Lemmas~\ref{upBound2Lem}--\ref{lowBound2bLem}, we show how to deduce from them a contradiction.

\begin{lemma}
\label{contr2Lem}
There exists $m_0'\ge1$ such that for all $m\ge m_0'$, all sufficiently small $\delta,\delta'\in(0,1)$ and all sufficiently large $M\ge1$ it holds that $\P\big(A_M^{(2)}\big)=0$.
\end{lemma}
\begin{proof}
We start by showing $\P(A_M^{(2,a)})=0$. Under the event $A_M^{(2,a)}$ we have $\xi\le1+4\delta$, so that Lemmas~\ref{upBound2Lem} and~\ref{lowBound20Lem}  give that 
$$\mu_+(m+m^{-1}+24\delta m)\ge 1-\eta +\mu_-(m-1+\sqrt{\eta^2+1}-20\delta m).$$
Hence,
\begin{align*}
m(\mu_+-{\mu_-})+{\mu_+}(m^{-1}+44\delta m)\ge1-\eta+\mu_-(\sqrt{\eta^2+1}-1)\ge\tfrac14,
\end{align*}
as required.

Next, we prove $\P(A_M^{(2,b)})=0$. Again, using Lemmas~\ref{upBound2Lem},~\ref{lowBound20Lem} and~\ref{lowBound2bLem} shows that
\begin{align*}
\mu_+\big (m+\xi-1+m^{-1}+42m\delta\big)\ge\mu_-(m-1+\sqrt{\eta^2+1})+\sqrt{2}(\xi-1)+|2-\xi-\eta|.
\end{align*}
After rearranging terms, we arrive at 
\begin{align*}
m({\mu_+}{\mu_-^{-1}}-1)+{\mu_+}{\mu_-^{-1}}\big(m^{-1}+42m\delta\big)\ge -1+\sqrt{\eta^2+1}+\tfrac{\sqrt{2}-\mu_+}{\mu_-}(\xi-1)+\tfrac{1}{\mu_-}|2-\xi-\eta|.
\end{align*}
First, the left-hand side becomes arbitrarily small if first $m$ is chosen sufficiently large and then $\delta,\delta'\in(0,1)$ are chosen sufficiently small. Next, it is shown in~\cite{nicolasDel} that $\mu\le \tfrac{35}{3\pi^2}<\sqrt{2}$ holds in the Delaunay case, so that the coefficient $(\sqrt{2}-\mu_+)/\mu_-$ is bounded away from $0$. Now, we can distinguish between the three cases i) $|\eta|\ge1/4$ ii) $\xi\ge5/4$ and iii) $|\eta|\le 1/4$ and $\xi\le5/4$ to see that the right-hand side remains bounded away from $0$.
\end{proof}

Finally, we provide the proofs of Lemmas~\ref{upBound2Lem}--\ref{lowBound2bLem}.

\begin{proof}[Proof of Lemma~\ref{upBound2Lem}]
Choose a subsquare center $v_i$ such that $\pi_2(v_i)=M+5.5\varepsilon M$ and $|\pi_1(v_i)+\xi M|\le\varepsilon M$.
Then, by Lemma~\ref{rmmConfLem2}, we obtain that $\ell^{(-M)}(P_1,P_1')\le\sqrt{M}$,
where $P_1'=q(v_i)$. Next, choose a subsquare center $v_j$ such that $\pi_2(v_j)=\pi_2(v_i)$ and 
$$|\pi_1(v_j)-(-1+\delta)M|\le\varepsilon M.$$
Then putting $P_3=q(v_j)$, condition {\bf (D1)} implies that $\ell^{(-M)}(P_1',P_3)\le\mu_+(\xi-1+\delta +2\varepsilon)M$. Finally, an application of {\bf (D2)} results in 
\begin{align*}
\ell^{(-M)}(P_3,Q)\le\mu_+(\sqrt{m^2+1}+18\delta )M\le\mu_+(m+m^{-1}+18\delta)M,
\end{align*}
as required.
\end{proof}

\begin{proof}[Proof of Lemma~\ref{lowBound20Lem}]
 First, note that $\eta\le1$ as $\gamma$ does not hit the set $[-(1-\delta)M,0]\times(\R\setminus[-M,M])$. 
 In particular,
\begin{align*}
\ell^{(nM)}(P_1,P_2)\ge(1-\eta)M,
\end{align*}
which proves the first claim.

Second, proceeding as in Lemma~\ref{lowBound10Lem}, condition {\bf (D2)} implies that 
$$\ell^{(nM)}(P_2,P)\ge\mu_-(\sqrt{\eta^2+1}-2\delta)M$$
and 
$$\ell^{(nM)}(P,Q)\ge\mu_-(m-1-18\delta m)M.$$
Hence, 
$$\ell^{(nM)}(P_2,Q)=\ell^{(nM)}(P_2,P)+\ell^{(nM)}(P,Q)\ge\mu_-(m-1+\sqrt{\eta^2+1}-20\delta m)M,$$
as required.
\end{proof}

\begin{proof}[Proof of Lemma~\ref{lowBound2bLem}]
First, Corollary~\ref{diagDelCor} and Remark~\ref{diagDelRem} show that the first edge in $\gamma[P_1]$ is one of the diagonal edges shown in Figure~\ref{deleteFig}. Hence by Corollary~\ref{diagDelCor}, writing $P_1'$ for intersection point of this edge with the vertical line $\ellv_{-M}$, we arrive at
\begin{align*}
\ell^{(nM)}(P_1,P_2)&\ge\ell^{(nM)}(P_1,P_1')+\ell^{(nM)}(P_1',P_2)\ge\sqrt{2}(\xi-1)M+|\pi_2(P_1')-\pi_2(P_2)|-\delta M.
\end{align*}
Since the right-hand side is at least $(\sqrt{2}(\xi-1)+|2-\xi-\eta|-2\delta)M$, we conclude the proof.
\end{proof}